\documentclass[11pt]{article}

\usepackage{amsthm,amsfonts,amssymb,amsmath,color}
\usepackage{bm}
\usepackage{bbm}
\usepackage[utf8x]{inputenc}
\usepackage{enumerate}
\usepackage[english]{babel}
\usepackage{esint}
\usepackage[mathscr]{eucal}
\usepackage{enumitem}
\usepackage[numbers,sort&compress]{natbib}
\usepackage{hyperref}
\hypersetup{
	colorlinks = true, %Colours the links instead of ugly red boxes
	urlcolor = blue, %Colour for external hyperlinks
	linkcolor = blue, %Colour of internal links
	citecolor = blue %Color of cites
}

\usepackage{url}

\usepackage{layout}
\allowdisplaybreaks

\numberwithin{equation}{section}

\renewcommand\a{\alpha}
\renewcommand\b{\beta}

\newcommand\s{\sigma}
\renewcommand\t{\tau}
\newcommand\R{\mathbb R}

\def\t{\tau}
\def\O{\Omega}

\def\l{\lambda}

\def\bvp{\bm{\varphi}}

\def\e{\varepsilon}

\newcommand\br{\begin{rem}}
	\newcommand\er{\end{rem}}
\newcommand\bp{\begin{pmatrix}}
	\newcommand\ep{\end{pmatrix}}
\newcommand\be{\begin{equation}}
	\newcommand\ee{\end{equation}}
\newcommand\ba{\begin{equation}\begin{aligned}}
		\newcommand\ea{\end{aligned}\end{equation}}

\newcommand\nn{\nonumber}

\usepackage{geometry}
\newcommand{\GG}{{\mathbb G}}

\newcommand{\HH}{{\mathbb H}}

\newcommand{\KK}{{\mathbb K}}

\newcommand{\supp}{{\rm supp }}

\newcommand{\uu}{{\mathbf u}}

\newcommand{\ww}{{\mathbf w}}

\newcommand{\vu}{\vc{u}}

\newcommand{\vc}[1]{{\bf #1}}

\newcommand{\Grad}{\nabla_x}

\newcommand{\dx}{\, {\rm d} {x}}

\newcommand{\dive}{{\rm div}}

\let\pa=\partial
\newcommand{\dd}{{\mathrm{d}}}
\newtheorem{defi}{Definition}[section]
\newtheorem{theorem}[defi]{Theorem}
\newtheorem{proposition}[defi]{Proposition}
\newtheorem{lemma}{Lemma}
\newtheorem{corollary}[defi]{Corollary}
\newtheorem{remark}{Remark} 

\newtheorem{assumption}[defi]{Assumption}

\newcommand{\calB}{\mathcal{B}}
\renewcommand{\l}{\langle}
\renewcommand{\r}{\rangle}

\begin{document}
	
	\title{Homogenization of the Navier--Stokes--Cahn--Hilliard system in the small-hole regime}
	
	\author{Yong Lu\footnote{School of Mathematics, Nanjing University, Nanjing 210093, China, luyong@nju.edu.cn}\and 
	Jiaojiao Pan\footnote{ Institute of Applied Physics and Computational Mathematics, Beijing 100088, China, panjiaojiao.math@gmail.com}\and 
	Luqi Wang\footnote{School of Mathematics, Nanjing University, Nanjing 210093, China, wangluqi@nju.edu.cn}}
	
	\date{}
	
	\maketitle
	
	\renewcommand{\refname}{References}
	
	\begin{abstract}
		This paper investigates the homogenization of the 3D Navier--Stokes--Cahn--Hilliard (NSCH) system in domains containing a large number of solid obstacles (named holes). Each hole has diameter of order $\varepsilon^{\alpha} \ (\alpha>3)$, where $\varepsilon > 0$ denotes the small length scale for inter-hole separation. Both viscosity and mobility depend on the phase-field variable. We establish two distinct asymptotic regimes: if the capillary strength $\lambda_\e\to \lambda>0$ as $\varepsilon\to 0$, the limit system coincides with the original NSCH system; if $\lambda_\e\to 0$ as $\varepsilon\to 0$, the scaled velocity, phase field and chemical potential converge to a weak solution to a  Stokes--Cahn--Hilliard (SCH) system. To the best of our knowledge, this work constitutes the first rigorous homogenization analysis for evolutionary NSCH flows with phase-dependent viscosity and mobility under the subcritical hole scaling.
	\end{abstract}
	

	%%%%%%%%%%%%%%%%%%%%%%%%%%%%%%%%%%%%%%%%%%%%%%%%%%%%%%%%%%%%%%%%%%%%%%%%%%%%%%%%%%%%%%%%%%
	
	\section{Introduction}\label{INTRODUCTION}
	Diffuse-interface (phase-field) models provide a rigorous framework for interfacial dynamics in multiphase flows. Unlike classical sharp-interface two-phase models requiring explicit surface tracking, the phase-field approach introduces an order parameter $\phi\in[-1,1]$ to smoothly distinguish two immiscible fluids across a thin interfacial transition layer, with $\phi=\pm 1$ for each phase. The temporal variation of order parameter follows the Cahn--Hilliard equation derived from free-energy functionals. The mass diffusion flux in this fourth-order evolution law is controlled by a phase-dependent mobility function $M(\phi)$, which describes the transport rate of the order parameter across fluid interfaces. 

When combined with the incompressible Navier--Stokes equations governing fluid motions, the phase-field model yields the thermodynamically consistent Navier--Stokes--Cahn--Hilliard (NSCH) system \cite{Abe-09}. 
Derived from the classic Model H proposed by Hohenberg and Halperin \cite{HH77}, this coupled framework bridges macroscopic hydrodynamic behaviors and microscopic phase separation.  
Abels, Garcke and Gr\"un \cite{AGG-12} systematically constructed its extensions with phase-dependent viscosity and variable mobility. The fundamental theories on global weak solutions and long-time dynamics of NSCH systems were developed in \cite{Abe-09,GG-10}, which form the theoretical basis of this work.

Many researchers investigated various variants of the NSCH system, including compressible and incompressible flows, fluids with variable densities, anisotropic interfacial energies, or degenerate mobility. Readers may refer to works \cite{Abe-09, Abe-09-long, Abe-09-CMP,  ADG13, AF08, AGG-12, Boy-99, GG-2D, GG-10, GGW-19} and the bibliographies therein for further research.
Most previous studies on the NSCH system have focused on the existence of weak solutions and long-time dynamics, whereas homogenization problems have received comparatively little attention.

Homogenization problems aim to find the effective macroscopic flow model for fluids evolving in domains containing a large number of solid obstacles or holes. As the number of holes tends to infinity while their characteristic size tends to zero, the effective large-scale flow model is called homogenized system. 	
The first mathematical work on this topic was done by Tartar \cite{Tartar1}. He studied the steady incompressible Stokes equations, where hole size matches the distance between holes, and derived Darcy's law.
Later, Allaire \cite{ALL-NS1, ALL-NS2} carried out a complete analysis for steady Stokes and Navier--Stokes equations in $\mathbb{R}^d \ (d\geq 2)$, and proved that the homogenized system depends entirely on the ratio $\s_{\e}$, defined by
\ba \label{sigma}
\sigma _\e:= \left(\frac{\e^d}{a_\e^{d-2}}\right)^{\frac{1}{2}},  \ d \geqslant 3;\quad{\sigma _\e}:= \e \left| \log \frac{{a_\e }}{\e} \right|^{\frac{1}{2}}, \ d = 2,
\ea
where $\e > 0$ stands for the distance between neighboring holes, and $a_\e > 0$ represents radius of each hole. The asymptotic behavior is divided into three cases based on the limit of $\s_{\e}$ as $\e\to 0$:
\begin{itemize}
\item Supercritical case: $\lim_{\e\to 0}\s_{\e}=0$. Holes are quite large and fluid velocity vanishes in the original scaling. Darcy’s law appears under proper rescaling.
\item Subcritical case: $\lim_{\e\to 0}\s_{\e}=\infty$. Holes are very tiny and the limit system is the same as the original one.
\item Critical case: $\lim_{\e\to 0}\s_{\e}=\s_{*}\in (0,+\infty)$. The holes induce an extra friction term in asymptotic limit, known as Brinkman's law.
\end{itemize}
The first author proposed a unified approach to all three cases above in \cite{L}.
	
	Homogenization of different types of single-phase fluid flows in perforated domains has been studied during the last decades. For compressible fluids, Masmoudi \cite{Mas-Hom} established the first rigorous homogenization result under the scaling where hole size matches inter-hole separation, and derived Darcy’s law.
H\"{o}fer, Ne\v{c}asov\'a and Oschmann \cite{HNO25} proved a quantitative version of this convergence result.
	A series of studies \cite{BO23,DFL,FL1,Lu-Schwarz18, OschmannPokorny2023,NO,NP} treated the subcritical case, and showed that the limit system stays identical to the original one.
	For the evolutionary incompressible Navier--Stokes system, Mikeli\'{c} \cite{Mik91} considered the case where the diameter of holes is proportional to the mutual distance, and obtained Darcy's law. Subsequent contributions \cite{FeNaNe, LY} covered the remaining scaling regimes. The homogenization of the inhomogeneous incompressible Navier--Stokes system can be found in \cite{LuPanYang2025,Pan2025,BOP26}.
		
	Moreover, extensive research addresses homogenization problems for complex fluid models in perforated domains. Feireisl, Novotn\'y and Takahashi \cite{FNT-Hom} investigated the supercritical case of the full Navier--Stokes--Fourier system; the subcritical case of stationary compressible Navier--Stokes--Fourier equations was treated in \cite{Lu-Pokorny}. As for inhomogeneous incompressible heat conducting fluid,  Feireisl, the first author and Sun \cite{FLS} analyzed the critical scaling regime in three-dimensional bounded domains and proved that the limit system obeys Brinkman’s law.
More recently, in \cite{LY-Qian, HLO2025}, with collaborators the first author considered evolutionary incompressible Carreau-type non-Newtonian viscous flows, and derived Darcy's law for supercritical perforations. 
	
         Despite abundant results for single-phase fluid homogenization in perforated domains, phase-field coupled multiphase flows remain largely unexplored in porous geometries. For stationary Stokes--Cahn--Hilliard (SCH) system, rigorous upscaling results were established in \cite{LM}, while formal asymptotic expansions for low-Reynolds two-phase porous flows can be found in \cite{DR15}. 
        For NSCH system, Metzger and Knabner \cite{MetKna} used asymptotic expansion and phase-field approach to derive Darcy' law.
        Recently, Chakrabortty, Dutta and Mahato \cite{CDM} used the periodic unfolding operator to derive the two-scale limit system, where the diameter of holes is proportional to mutual distance. No prior rigorous work addresses the evolutionary NSCH system with phase-dependent viscosity and mobility under the subcritical hole scaling. 
         
          In this paper, we investigate the homogenization of an incompressible NSCH system in perforated domains in $\mathbb{R}^{3}$ with very tiny holes, whose diameter is of order $\e^{\a} \ (\a>3)$, and the mutual distance is of order $\e$, corresponding to the subcritical case. We only consider non-degenerate mobility case, consistent with the standard analytical setup in \cite{Abe-09, GG-10}. The homogenization of NSCH system with degenerate mobility of the type studied in \cite{ADG13} remains open.
          
          In contrast to stationary SCH models, the evolutionary NSCH system introduces the essential analytical difficulties:  since the convective term $\dive(\vu_\e \otimes \vu_\e)$ and some terms involving $\phi_\e$ are nonlinear, standard weak convergence is insufficient to pass these nonlinear terms to the homogenized limit. To overcome this difficulty, we establish weak convergence of convective term and strong convergence of the extended phase-field variable $\widetilde\phi_\e$, via refined compactness arguments. Two distinct asymptotic regimes are distinguished based on capillary strength:
          
        \begin{itemize}
        \item If  $\lambda_\e\to \lambda>0$ as $\varepsilon\to 0$, then the limit system coincides with the original NSCH system. 
        \item If $\lambda_\e\to 0$ as $\varepsilon\to 0$, then the asymptotic limit behavior of scaled velocity, phase-field variable and chemical potential is governed by a Stokes--Cahn--Hilliard (SCH) system.
           \end{itemize}

	\paragraph{{\bf Organization of the paper:}} The paper is organized as follows. In Section~\ref{INTRODUCTION}, we introduce the problem under consideration, give the definition of weak solution, present several lemmas and basic inequalities, and formulate the main results. 
	Section~\ref{sec:unifbds} is devoted to obtaining uniform bounds on the velocity, phase-field variable and chemical potential. 
	In Section~\ref{Equation in homogenized domain}, we establish the momentum equations, phase-field equation and chemical-potential equation in homogenized domain.
	To pass $\e\to 0$, we study the convergence of nonlinear convective term in Section~\ref{Convergence of the nonlinear convective term} and the strong convergence of $\widetilde\phi_\e$ in Section~\ref{Strong convergence of phi}, via Aubin--Lions type compactness arguments. Then in Section~\ref{Homogenization process}, we pass $\e\to 0$ to derive our main results. At the end, we study the vanishing capillary parameter regime in Section~\ref{lambda0}.
	
		\paragraph{{\bf Notations:}} In this paper, we write $\nabla=\Grad$, $\dive=\dive_x$, $\Delta=\Delta_x$, and use standard notations $L^p$ and $W^{k,p}$ for Lebesgue and Sobolev spaces, respectively. Specially, we denote $H^k=W^{k,2}$.
Moreover, for a domain $D \subset \R^3$, we use $L^{q}_0(D)$ to denote the space of functions in $L^q(D)$ with zero integral mean:
	\be
	L^q_0(D):=\left\{f\in L^q(D): \, \int_{D} f\,\dx=0\right\}. \nn
	\ee
We use $W^{1,q}_0$ to denote functions in $W^{1,q}$ with zero trace. 
For vector-valued fields, we introduce solenoidal subspaces of $L^q$ and $W^{1,q}$ as follows:
\ba\nn
&L^q_{\sigma}(D;\mathbb{R}^3):=\left\{\bm{f}\in L^q(D;\mathbb{R}^3): \, \dive \bm{f}=0  \right\}, \\
&W^{1,q}_{\sigma}(D;\mathbb{R}^3):=\left\{\bm{f}\in W^{1,q}(D;\mathbb{R}^3): \, \dive \bm{f}=0  \right\}. 
\ea
The dual spaces of $W^{1,q}$ and $W^{1,q}_\sigma$ are denoted by $W^{-1,q'}$ and $W^{-1,q'}_\sigma$, respectively.
We write $\|\cdot\|_{L^p L^q}=\|\cdot\|_{L^p(0,T;L^q(D))}$ and $\|\cdot\|_{L^p W^{k,q}}=\|\cdot\|_{L^p(0,T;W^{k,q}(D))}$ without further notice. We omit the time interval $(0,T)$, spatial domain $D=\Omega_\e$ or $\Omega$, and time/spatial variable when there is no confusion. The inner product of two vectors $\mathbf a,\mathbf b \in \R^{3}$ is denoted by $\mathbf a\cdot\mathbf b$; the inner product of two matrices $A,B \in \R^{3 \times 3}$ is defined as $A:B = \sum_{i, j=1}^3 A_{ij} B_{ij}$. 

Without loss of generality, we use $C$ to denote a positive constant independent of $\e$.
All weak convergences hold up to a subsequence, and we omit explicit mention of this subsequence for brevity.
		
\subsection{Problem formulation}\label{Problem formulation}
In this paper, we focus on the three-dimensional case in \eqref{sigma} and take 
\be\nn
a_{\e}=\e^{\a} \quad\mbox{with}\quad \a>3,
\ee
which yields $\s_{\e}=\e^{\frac{3-\a}{2}} \to \infty$ as $\e \to 0$. Let $\O \subset \R^3$ be a bounded domain of class $C^{2,\b}$ with $0<\b<1$. We define a family of perforated domains $\{ \O_\e \}_{\e>0}$ as follows: 
	\begin{align}\label{defi-holes-1}
	\O_\e = \O \setminus \bigcup_{k\in K_\e} T_{\e, k},\ && K_\e:=\{k\in \mathbb{Z}^3:\ \e \overline{Q}_{k} \subset \O \}, && Q_{k} := \Big( -\frac{1}{2},\frac{1}{2} \Big)^3 + k,\ \ k\in \mathbb{Z}^3.
	\end{align}
	Each set $T_{\e,k}$ stands for a single hole (or obstacle). We impose the following assumption on the hole geometry:
	\be\nn
	T_{\e,k} := x_{\e,k} + a_\e T_0 \subset\subset B(x_{\e,k}, b_0a_{\e} ) \subset\subset B(x_{\e,k}, b_1 a_{\e} ) \subset\subset \e Q_{k} \subset \O, 
	\ee
	with
	\be\nn
	|x_{\e, k} - x_{\e, l}| \geq 2\e,\quad \ \forall \ k \neq l \in K_\e. 
	\ee
	Here, $ T_0 \subset \R^3$ is a bounded, simply connected domain of class $C^{2,\b}$ contained in $Q_{0}$. The notation $B(x,r)$ denotes the open ball centered at $x$ with radius $r$, $b_0$ and $b_1$ are positive constants. 	Each $T_{\e,k}$ has diameter of order $\mathcal{O}(a_\e)$ and the minimal distance between distinct hole centers is of order $\mathcal{O}(\e)$. 
	The number of holes contained in $\O$ has the bound
	\be\label{defi-holes-2}
	|K_\e| \leq C(\e^{-3} |\Omega| + o(1)),\quad \text{for some} ~C>0~ \text{independent of} ~ \e.
	\ee
	
We study a diffuse-interface model for a binary incompressible mixture in the perforated domain $\O_{\e} \subset \R^3$.
The velocity $\vu_\e$, the hydrodynamic pressure $p_\e$, the phase-field variable (order parameter) $\phi_\e$ and the chemical potential $\mu_\e$ satisfy the following NSCH system in $(0,T) \times \O_{\e}$:
\ba\label{NSCH}
	\begin{cases}		
		\partial_t  \vu_\e  +\dive(\vu_\e\otimes \vu_\e)
		-\dive(\nu(\phi_\e) D(\vu_\e)) + \nabla p_\e +\lambda_\e \phi_\e \nabla \mu_\e=\mathbf g_\e,  &\text{in } (0,T) \times \O_{\e},\\
		 \dive \vu_{\e} = 0, &\text{in } (0,T) \times \O_{\e},\\
		\partial_t \phi_\e +\vu_\e\cdot\nabla\phi_\e-\dive( M(\phi_\e)\nabla\mu_\e)
		=0 , & \text{in } (0,T) \times \O_{\e},\\
		\mu_\e=-\Delta \phi_\e+F'(\phi_\e) ,& \text{in } (0,T) \times \O_{\e},\\
		\vu_{\e}(0,x) = \vu_{\e}^{0}(x),\quad \phi_\e(0,x)=\phi_\e^0(x) & \text{in } \O_\e.
	\end{cases}
	\ea
Here, the rate-of-strain tensor $D(\uu_\e)$ is defined by $D(\uu_\e)=\frac{1}{2}(\nabla \uu_\e+\nabla^{\textup{T}} \uu_\e)$. 
The fluid is contained in a bounded domain $\Omega_\e$, on the boundary of which the velocity, phase-field variable and chemical potential obey the boundary condition
\begin{equation} \label{boundary condition}
\vu_{\e} = \mathbf{0}, \quad \mathbf{n_\e}\cdot \nabla\phi_\e=0, \quad 
\mathbf{n_\e}\cdot(M(\phi_\e) \nabla\mu_\e)=0 \quad \text{ on }\pa\Omega_\e.
\end{equation}
Here, $\mathbf{n_\e}$ denotes unit outward normal vector on $\pa\Omega_\e$. 
The viscosity function $\nu: \mathbb{R}\to\mathbb{R}^+$ is Lipschitz continuous and satisfies
\ba\label{bound nu}
0<\nu_1<\nu(s)<\nu_2, \quad s\in\mathbb{R} ,
\ea
where $\nu_1$ and $\nu_2$ are positive constants.
The Cahn--Hilliard equation contains a phase-dependent mobility function $M(\phi_\e)$ governing the diffusion flux of the order parameter, where $M: \mathbb{R}\to \mathbb{R}^+$ is Lipschitz continuous and satisfies
\ba\label{bound M}
0<M_1<M(s)<M_2, \quad s\in\mathbb{R},
\ea
for some positive constants $M_1$ and $M_2$.
The smooth double-well potential $F:\mathbb{R}\to \mathbb{R}$ is given by
\ba\label{Fs}
F(s)=\frac{(s^2-1)^2}{4}.
\ea
It is straightforward to check that
\ba\label{F bound}
\frac{1}{2}s^2-\frac{3}{4}\leq F(s)\leq \frac{1}{4}s^4+\frac{1}{4}, \quad s\in \mathbb{R},
\ea
and the derivative satisfies
\ba\nn
F'(s)=s^3-s,\quad
|F'(s)|\leq |s|^3+1, \quad s\in \mathbb{R}.
\ea

The uniform positive lower and upper bounds on $\nu$ and $M$ ensure uniform ellipticity of the viscous and diffusion operators, which is essential to derive a priori estimates. The explicit double-well potential $F$ corresponds to the standard Ginzburg--Landau free energy density for binary phase flows. 

We now impose the scaling assumptions governing the capillary strength, initial data and external force. These scaling assumptions enable us to derive uniform estimates and apply Aubin--Lions compactness arguments. Under proper scaling, we can distinguish two distinct asymptotic regimes depending on whether the limiting capillary strength vanishes or not.

\begin{assumption}\label{Assumption}
We impose the following assumptions.
\begin{itemize}
\item The capillary strength $\lambda_\e>0$ depends on $\e$ and satisfies the limit relation
\be\label{assumption-lambda}
 \lim_{\e\to 0} \lambda_\e=\lambda\in[0,+\infty).
\ee

\item The initial values satisfy
\ba\label{initial-u-phi}
\uu_\e^0=\sqrt{\lambda_\e} \uu_0,\quad \phi_\e^0=\phi_0 \quad \text{in }\Omega,
\ea
where $\uu_0\in L^2(\Omega;\mathbb R^3)$ and $\phi_0\in H^1(\Omega)$, $-1\leq \phi_0\leq 1$.

\item There exists $\mathbf g\in L^2((0,T)\times \Omega; \mathbb R^3)$ such that
\ba\label{assumption-g}
\mathbf g_\e=\sqrt{\lambda_\e}\mathbf g \quad \text{ in } (0,T)\times \Omega .
\ea
\end{itemize}
\end{assumption}
\begin{remark}
The assumptions of initial values \eqref{initial-u-phi} and external force \eqref{assumption-g} can be replaced by
\ba\nn
&\frac{\uu^0_\e}{\sqrt{\lambda_\e}}\to \uu_0 \text{ strongly in } L^2(\Omega; \mathbb{R}^3), \quad \phi_\e^0\to \phi_0  \text{ strongly in }H^1(\Omega),\\
&\frac{\mathbf g_\e}{\sqrt{\lambda_\e}}\to \mathbf g \text{ weakly in } L^2((0,T)\times \Omega; \mathbb{R}^3).
\ea
For the convenience of analysis, we adopt assumptions \eqref{initial-u-phi} and \eqref{assumption-g} .
\end{remark}

\subsection{Weak solutions}\label{Weak solutions}

Let us introduce the definition of finite energy weak solutions to NSCH system \eqref{NSCH}--\eqref{boundary condition}.

\begin{defi}[weak solution]\label{def weak solution}	
A triple $(\vu_{\e},\phi_\e,\mu_\e)$ is called a \textbf{(finite energy) weak solution} to problem \eqref{NSCH}--\eqref{boundary condition} in $(0,T) \times \O_{\e}$ provided:	
\begin{itemize}
				
\item the $(\vu_{\e},\phi_\e,\mu_\e)$ satisfies:
\ba \nn
&\vu_\e \in L^\infty(0,T; L^2_{\sigma}(\Omega_\e; \R^3)) \cap 
L^2(0,T; H^1_0(\Omega_\e; \R^3)), 
\pa_t \vu_\e\in L^{4\over3}(0,T; H^{-1}(\Omega_\e; \R^3)),\\
&\phi_\e \in L^\infty(0,T; H^1(\Omega_\e)), 
\pa_t \phi_\e\in L^2(0,T; H^{-1}(\Omega_\e)),\\
&\mu_\e \in L^2(0,T; H^1(\Omega_\e));
\ea

\item the momentum equation: for any ${\bm\varphi}\in C_c^\infty([0,T)\times \Omega_\e;\mathbb{R}^3)$ with $\dive {\bm\varphi}=0$, there holds
\ba\label{def-weak-momentum}
&-\int_0^T \int_{\Omega_\e} \uu_\e \cdot\pa_t \bm\varphi \dd x \dd t
 -\int_0^T\int_{\Omega_\e} \uu_\e\otimes\uu_\e:\nabla \bm\varphi \dd x \dd t
 +\int_0^T\int_{\Omega_\e} \nu (\phi_\e) D(\uu_\e):D(\bm\varphi)\dd x \dd t\\
&\quad
+ \int_0^T\int_{\Omega_\e}\lambda_\e \phi_\e \nabla \mu_\e\cdot \bm\varphi\dd x \dd t= \int_0^T\int_{\Omega_\e} \mathbf g_\e\cdot \bm\varphi \dd x \dd t +\int_{\Omega_\e}  \uu_\e^0 \cdot \bm\varphi(0) \dd x;
\ea

\item the phase-field equation: for any ${\varphi}\in C_c^\infty([0,T)\times \Omega_\e)$, 
\ba\label{def-weak-phase}
&-\int_0^T\int_{\Omega_\e} \phi_\e\pa_t\varphi   \dd x\dd t+
 \int_0^T\int_{\Omega_\e} M(\phi_\e)\nabla\mu_\e\cdot\nabla \varphi \dd x \dd t\\
&\quad= \int_0^T\int_{\Omega_\e}  \phi_\e \uu_\e\cdot\nabla \varphi \dd x \dd t
+\int_{\Omega_\e} \phi_\e^0\varphi(0)   \dd x;
\ea

\item the chemical-potential equation: for any $\psi\in C_c^\infty((0,T)\times \Omega_\e)$, 
\ba\label{def-weak-chemical}
\int_0^T\int_{ \Omega_\e}  \mu_\e \psi \dd x \dd t=
\int_0^T\int_{ \Omega_\e} \nabla\phi_\e \cdot \nabla\psi \dd x\dd t+
\int_0^T\int_{ \Omega_\e}  F'(\phi_\e) \psi \dd x\dd t;
\ea

\item for a.a. $t\in[0,T)$, there holds the energy inequality
\ba\label{energy inequality}
&\mathcal E_\e(\vu_\e(t),\phi_\e(t))+\int_{0}^{t}\int_{\Omega_\e}(  \nu(\phi_\e) |D(\vu_\e)|^2+\lambda_\e M(\phi_\e)|\nabla \mu_\e|^2 )\dd x\dd \t\\
&\quad\leq \mathcal E_\e(\vu_\e(0),\phi_\e(0))+\int_{0}^{t}\int_{\Omega_\e}  \mathbf g_\e \cdot \vu_\e\dd x \dd \t,
\ea
where
\ba\nn
\mathcal E_\e(\mathbf v,\psi)=\frac{1}{2}\|\mathbf v\|_{L^2(\O_\e)}^2+\frac{\lambda_\e}{2}\|\nabla \psi\|_{L^2(\O_\e)}^2+\lambda_\e\int_{\Omega_\e} F(\psi) \dd x.
\ea

\end{itemize}
\end{defi}	

\medskip

\begin{remark}
For any fixed $\e>0$, the existence of weak solution was shown by Abels \cite{Abe-09} for $\mathbf g_\e= \mathbf 0$ and Gal and Grasselli \cite{GG-10} for $\mathbf g_\e\in L^2(0,T;H^{-1}(\Omega;\mathbb R^3))$. In \cite{GG-10}, they also showed that
\ba\nn
&\vu_\e\in C([0,T];W^{-1,2}_\sigma(\Omega_\e) )\cap C_w([0,T];L^2_\sigma(\Omega_\e) ),\\
&\phi_\e\in L^2(0,T; H^3(\Omega_\e))\cap C([0,T]; L^2(\Omega_\e)).
\ea
Hence, the initial conditions make sense.
\end{remark}

\subsection{Preliminary}

In this section, we collect several lemmas and inequalities that will be frequently used throughout the subsequent analysis.

For $\phi_\e$ and $\mu_\e$, boundary condition in \eqref{boundary condition} implies that zero extension would destroy the Sobolev regularity. To resolve this issue, we introduce the following extension lemma (see Lemma 3 in \cite{CS79} and Lemma 3.1 in \cite{Lu-Pokorny}):
\begin{lemma}[Extension operator]\label{extension lemma}
There exists an extension operator $ E^\e: H^1(\Omega_\e)\to H^1(\Omega)$, such that
\be\nn
E^\e f=f \quad\text{in } \Omega_\e,
\ee
\be\nn
\|E^\e f\|_{H^1(\Omega)}\leq C\| f\|_{H^1(\Omega_\e)},
\ee
where $C$ is a constant independent of $\e$. 
\end{lemma}

To derive the limit momentum equations, we recall the following lemma, which introduces a Bogovskii type operator (see Theorem 2.3 in \cite{DFL} and Proposition 2.2 in \cite{Lu-Schwarz18}):

\begin{lemma}[Bogovskii type operator]\label{lem-div}
Let $\Omega_\e$ be defined through \eqref{defi-holes-1}--\eqref{defi-holes-2} with $a_{\e} = \e^{\alpha}$, $\alpha\geq 1$. Then there exists a linear operator
\ba\nn
\calB_\e : L^{q}_0(\Omega_\e) \to W_0^{1,q}(\Omega_\e; \R^3),\ 1 < q < \infty,
\ea
such that for any $f\in L_0^{q}(\Omega_\e)$,
\ba\label{Bef}
\dive \calB_\e(f) =f \ \text{in} \ \Omega_\e,~~\|\calB_\e(f)\|_{W_0^{1,q}(\Omega_\e; \R^3)}\leq C\, \big(1+\e^{\frac{(3-q)\a-3}{q}}\big)\|f\|_{L^q(\Omega_\e)},
\ea
for some constant $C$ independent of $\e$.
\end{lemma}

There are also some basic inequalities in perforated domains as follows (see Chapter $11$ of \cite{FENO6}).
\begin{lemma}[Poincar\'e inequality] \label{Poincare inequality}
For $\mathbf v\in H_0^{1}(\Omega_\e)$, there holds
\be\nn
\|\mathbf v\|_{L^{2}(\Omega_\e)}\leq C\|\nabla \mathbf v\|_{L^{2}(\Omega_\e)},
\ee
where the constant is independent of $\e$.
\end{lemma}

\begin{lemma}[Korn inequality] \label{Korn inequality}
For $\mathbf v\in H_0^{1}(\Omega_\e)$, there holds
\be\nn
\|\nabla \mathbf v\|_{L^{2}(\Omega_\e)}\leq C\|D(\mathbf v)\|_{L^{2}(\Omega_\e)},
\ee
where the constant is independent of $\e$.
\end{lemma}

Moreover, the following Aubin--Lions type argument (see Lemma 5.1 in \cite{Lions-com}) is introduced to obtain the convergence of nonlinear convective term.
\begin{lemma}\label{lem-AL}
Let $g^n\to g$ weakly in $L^{p_1}(0,T;L^{p_2}(\Omega))$, $h^n \to h$ weakly in $L^{q_1}(0,T;L^{q_2}(\Omega))$,
where $1 \leq p_1,p_2 \leq +\infty$ and
\be\nn
\frac{1}{p_1}+\frac{1}{q_1}=\frac{1}{p_2}+\frac{1}{q_2}=1.
\ee
We assume in addition that
\be\nn
\frac{\pa g^n}{\pa t}
\text{ is bounded in }
L^1(0,T;W^{-m,1}(\Omega))
\text{ for some } m\geq 0 \text{ independent of } n,
\ee
and
\be\nn
\|h^n-h^n(\cdot+\xi,t)\|_{L^{q_1}(0,T;L^{q_2}(\Omega))}\to 0\quad \text{as } |\xi|\to 0,
\text{ uniformly in } n.
\ee
Then
\be\nn
g^n h^n \to gh \text{ in } {\cal D'}((0,T)\times \Omega).
\ee
\end{lemma}

\subsection{Main results}
Since $\vu_\e\in L^2(0,T; H^1_0(\Omega_\e; \R^3))$, we can extend $\vu_\e$ to be zero outside $\Omega_\e$, that is,
\ba\nn
\widetilde\vu_\e(t,x)=\left\{
\begin{aligned}
&\vu_\e(t,x), \quad &x\in \Omega_\e,\\
&\mathbf 0,\quad &x\in \Omega\backslash \Omega_\e.
\end{aligned}
\right .
\ea
For a.a. $t\in[0,T)$, we use extension operator $ E^\e$ in Lemma \ref{extension lemma} to extend $\phi(t,\cdot), \mu(t,\cdot)\in H^1(\Omega_\e)$ to 
\ba\nn
\widetilde{\phi}(t,\cdot):=E^\e(\phi(t,\cdot)), \quad \widetilde{\mu}(t,\cdot):=E^\e(\mu(t,\cdot)).
\ea

Now we are ready to formulate our main result:
\begin{theorem} \label{Main theorem}
Let $\{ \Omega_\e\}_{\e > 0}$ be a family of perforated domains specified in Section \ref{Problem formulation}. Let $(\vu_{\e},\phi_\e,\mu_\e)$ be a weak solution to the problem \eqref{NSCH}--\eqref{boundary condition}. Under the assumptions of capillary parameter, initial data and external force in \eqref{assumption-lambda}--\eqref{assumption-g}, up to a subsequence, the extension $( \widetilde\vu_\e, \widetilde\phi_\e, \widetilde\mu_\e)$ satisfy
\ba \nn
\widetilde\vu_\e \to \vu \ &\mbox{weakly-(*) in}\ L^\infty(0,T; L^2(\Omega; \R^3)), \mbox{ weakly in } L^2(0,T; H^1_0(\Omega;\R^3)), \\
&\mbox{and strongly in}\ L^2(0,T; L^2(\Omega;\mathbb{R}^3)),\\
\widetilde\phi_\e \to \phi \ &\mbox{weakly-(*) in}\ L^\infty(0,T; H^1(\Omega))  \mbox{ and strongly in}\ L^2(0,T; L^2(\Omega)),\\
\widetilde\mu_\e \to \mu \ &\mbox{weakly in}\ L^2(0,T; H^1(\Omega)),
\ea
where $( \vu, \phi,\mu)$ is a weak solution to the problem
\ba\label{limit eq}
	\begin{cases}		
		\partial_t  \vu + \dive(\vu\otimes \vu) -\dive(\nu(\phi) D(\vu)) + \nabla p +\lambda \phi \nabla \mu=\sqrt{\lambda}\mathbf g,  &\mbox{in } (0,T) \times \O,\\
		 \dive \vu = 0, &\mbox{in } (0,T) \times \O,\\
		\partial_t \phi +\vu\cdot\nabla\phi-\dive(M(\phi) \nabla\mu) 
		=0 , & \mbox{in } (0,T) \times \O,\\
		\mu=-\Delta \phi+F'(\phi) ,& \mbox{in } (0,T) \times \O,\\
		\vu(0,x) = \sqrt{\lambda}\vu_{0}(x),\quad \phi(0,x)=\phi_0(x) & \mbox{in } \O,
	\end{cases}
	\ea
with boundary condition
 \begin{equation} \label{boundary condition for limit eq}
\vu = \mathbf{0}, \quad \mathbf{n}\cdot \nabla\phi=0, \quad \mathbf{n}\cdot (M(\phi)\nabla\mu)=0 \quad\text{ on } \pa\Omega.
\end{equation}
Here, $\mathbf{n}$ denotes unit outward normal vector on $\pa\Omega$. 
\end{theorem}

\begin{remark}  We list several key observations regarding the weak formulation and asymptotic behavior of the coupled NSCH system.
\begin{enumerate}[label=(\roman*), leftmargin=*, align=left]
\item The definition of weak solution to problem \eqref{limit eq}--\eqref{boundary condition for limit eq} is similar to Definition \ref{def weak solution}.

\item Both the microscopic and macroscopic boundary conditions involve the mobility function $M(\phi)$, since the mass diffusion flux in the Cahn--Hilliard equation takes the form $\mathcal J=-M(\phi)\nabla \mu$, and no mass can pass through solid obstacles in perforated media or the boundary of the macroscopic domain.

\item If $\lambda>0$ in \eqref{assumption-lambda}, the problem \eqref{limit eq}--\eqref{boundary condition for limit eq} is the same as original NSCH system.

\item If $\lambda=0$ in \eqref{assumption-lambda}, then the strong convergence of $\widetilde \vu_\e$ in $L^2(0,T;L^2(\Omega))$ implies that $\vu=\bf 0$,  and $(\phi$, $\mu)$ is a weak solution to a Cahn--Hilliard system.
Under the scaling $\vu_\e=\sqrt{\lambda_\e}\ww_\e$, we can get uniform estimates for $\ww_\e$. Notice that in the momentum equations $\eqref{NSCH}_1$, the convective inertial term $\dive(\vu_\e\otimes \vu_\e)$ and capillary force term $\lambda_\e \phi_\e \nabla \mu_\e$ are of order $O(\lambda_\e)$, while other terms are of order $O(\sqrt{\lambda_\e})$, so convective inertial term and capillary force term disappear in the limit system.
We can show that the limit of the scaled extended velocity $\widetilde\ww_\e$ in $L^2(0,T;L^2(\Omega))$ is a weak solution to a Stokes equation and give the following corollary. 
\end{enumerate}
\end{remark}

\begin{corollary} \label{corollary}
Let $\{ \Omega_\e\}_{\e > 0}$ be a family of perforated domains specified in Section \ref{Problem formulation}. Let $(\vu_{\e},\phi_\e,\mu_\e)$ be a weak solution to the problem \eqref{NSCH}--\eqref{boundary condition}. Under the assumption of capillary parameter, initial data and external force in \eqref{assumption-lambda}--\eqref{assumption-g} with $\lambda=0$, up to a subsequence, the extension $( \widetilde\vu_\e, \widetilde\phi_\e, \widetilde\mu_\e)$ satisfy
\ba \nn
\frac{\widetilde\vu_\e}{\sqrt{\lambda_\e}} \to \ww \ &\mbox{weakly-(*) in}\ L^\infty(0,T; L^2(\Omega; \R^3)), \mbox{ weakly in}\ L^2(0,T; H^1_0(\Omega;\R^3)), \\
&\mbox{and strongly in}\ L^2(0,T; L^2(\Omega;\mathbb{R}^3)),\\
\widetilde\phi_\e \to \phi \ &\mbox{weakly-(*) in}\ L^\infty(0,T; H^1(\Omega))  \mbox{ and strongly in}\ L^2(0,T; L^2(\Omega)),\\
\widetilde\mu_\e \to \mu \ &\mbox{weakly in}\ L^2(0,T; H^1(\Omega)),
\ea
where $( \ww, \phi,\mu)$ is a weak solution to the problem
\ba\label{limit eq-coro}
	\begin{cases}		
		\partial_t  \ww  -\dive(\nu(\phi) D(\ww)) + \nabla p =\mathbf g,  &\mbox{in } (0,T) \times \O,\\
		 \dive \ww = 0, &\mbox{in } (0,T) \times \O,\\
		\partial_t \phi -\dive(M(\phi)\nabla \mu) 
		=0 , & \mbox{in } (0,T) \times \O,\\
		\mu=-\Delta \phi+F'(\phi) ,& \mbox{in } (0,T) \times \O,\\
		\ww(0,x) = \vu_{0}(x),\quad \phi(0,x)=\phi_0(x) & \mbox{in } \O,
	\end{cases}
	\ea
with boundary condition
 \begin{equation} \label{boundary condition for limit eq-coro}
\ww = \mathbf{0}, \quad \mathbf{n}\cdot \nabla\phi=0, \quad \mathbf{n}\cdot (M(\phi) \nabla\mu)=0 \quad\text{ on } \pa\Omega.
\end{equation}
\end{corollary}
\begin{remark}
The problem \eqref{limit eq-coro}--\eqref{boundary condition for limit eq-coro} is a Stokes--Cahn--Hilliard (SCH) system, whose weak solution is defined in a similar way as it in Definition \ref{def weak solution}.
\end{remark}

\section{Uniform bounds}\label{sec:unifbds}
In this section, we derive uniform estimates for the sequence $(\vu_\e,\phi_\e,\mu_\e)$.
From the energy inequality \eqref{energy inequality}, we have for a.a. $t\in[0,T)$,
\ba\nn
&\mathcal E_\e(\vu_\e(t),\phi_\e(t))
+\int_{0}^{t}\int_{\Omega_\e}(  \nu_1 |D(\vu_\e)|^2+\lambda_\e M_1 |\nabla \mu_\e|^2 )\dd x\dd \t\\
&\quad\leq \mathcal E_\e(\vu_\e(t),\phi_\e(t))
+\int_{0}^{t}\int_{\Omega_\e}(  \nu(\phi_\e) |D(\vu_\e)|^2+\lambda_\e M(\phi_\e) |\nabla \mu_\e|^2 )\dd x\dd \t\\
&\quad\leq \mathcal E_\e(\vu_\e(0),\phi_\e(0))
+\int_{0}^{t}\int_{\Omega_\e}  \mathbf g_\e \cdot \vu_\e \dd x\dd \t\\
&\quad\leq \mathcal E_\e(\vu_\e^0,\phi_\e^0)
+\int_{0}^{t}  \frac{a}{2}\|\mathbf g_\e(\t)\|_{L^2(\Omega_\e)}^2+ \frac{1}{2a}\|\vu_\e(\t)\|_{L^2(\Omega_\e)}^2 \dd \t\\
&\quad\leq \mathcal E_\e(\vu_\e^0,\phi_\e^0)
+\int_{0}^{t}  \frac{a}{2}\|\mathbf g_\e(\t)\|_{L^2(\Omega_\e)}^2+ \frac{C_1}{2a}\|D(\vu_\e)(\t)\|_{L^2(\Omega_\e)}^2 \dd \t.
\ea
Here, the total energy is given as
\be\nn
\mathcal E_\e(\vu_\e(t),\phi_\e(t)):=\frac{1}{2}\|\vu_\e(t)\|_{L^2(\O_\e)}^2+\frac{\lambda_\e}{2}\|\nabla \phi_\e(t)\|_{L^2(\O_\e)}^2+\lambda_\e\int_{\Omega_\e} F(\phi_\e(t)) \dd x,
\ee
and the constant $C_1$ is from Lemmas \ref{Poincare inequality} and \ref{Korn inequality}, which is independent of $\e$.
Taking $a=C_1/\nu_1$, we deduce that
\ba\nn
&\frac{1}{2}\|\vu_\e(t)\|_{L^2(\O_\e)}^2+\frac{\lambda_\e}{2}\|\nabla \phi_\e(t)\|_{L^2(\O_\e)}^2+\lambda_\e\int_{\Omega_\e} F(\phi_\e(t)) \dd x
+\int_{0}^{t}\int_{\Omega_\e}(  \frac{\nu_1}{2} |D(\vu_\e)|^2+ \lambda_\e M_1 |\nabla \mu_\e|^2 )\dd x\dd \t\\
&\quad\leq \frac{1}{2}\|\vu_\e^0\|_{L^2(\O_\e)}^2+\frac{\lambda_\e}{2}\|\nabla \phi_\e^0\|_{L^2(\O_\e)}^2+\lambda_\e\int_{\Omega_\e} F(\phi_\e^0) \dd x+  \frac{C_1}{2\nu_1}\|\mathbf g_\e\|_{L^2((0,T)\times\Omega_\e)}^2.
\ea
By \eqref{Fs}, there holds
\be\nn
\int_{\Omega_\e} F(\phi_\e(t)) \dd x \geq 0,
\ee
which combined with Assumption \ref{Assumption} gives
\ba\label{uni-bound}
&\|\vu_\e\|_{L^\infty L^2}+\|D(\vu_\e)\|_{L^2 L^2}\leq C \sqrt{\lambda_\e},\\
&\|\nabla\phi_\e\|_{L^\infty L^2}+\|\nabla\mu_\e\|_{L^2 L^2}+\int_{\Omega_\e} F(\phi_\e(t)) \dd x \leq C.
\ea
Combining \eqref{F bound} and \eqref{uni-bound}, we can deduce
\be\nn
\int_{\Omega_\e}|\phi_\e(t)|^2 \dd x\leq \int_{\Omega_\e}\frac{3}{4}+  F(\phi_\e(t)) \dd x\leq C,\quad \text{a.a. } t\in[0,T),
\ee
and thus
\be\label{phi e LinftyH1}
\|\phi_\e\|_{L^\infty H^1}\leq C.
\ee

Since $C_c^\infty((0,T)\times \Omega_\e)$ is dense in $L^2(0,T;H^1(\Omega_\e))$, \eqref{def-weak-chemical} holds for functions in $L^2(0,T;H^1(\Omega_\e))$. Taking $\mu_\e\in L^2(0,T;H^1(\Omega_\e))$ as a test function in \eqref{def-weak-chemical}, we derive
\ba\nn
\int_0^T\int_{ \Omega_\e}  |\mu_\e|^2 \dd x \dd t&=
\left|\int_0^T\int_{ \Omega_\e} \nabla\phi_\e \cdot \nabla\mu_\e \dd x\dd t+
\int_0^T\int_{ \Omega_\e}  F'(\phi_\e) \mu_\e \dd x\dd t \right|\\
&\leq \left|\int_0^T\int_{ \Omega_\e} \nabla\phi_\e \cdot \nabla\mu_\e \dd x\dd t\right|+  \frac{1}{2}\int_0^T\int_{ \Omega_\e}  |F'(\phi_\e)|^2\dd x\dd t+ \frac{1}{2}\int_0^T\int_{ \Omega_\e}  | \mu_\e|^2 \dd x\dd t,
\ea
which combined with \eqref{uni-bound} and \eqref{phi e LinftyH1}, implies that
\ba\nn
\|\mu_\e\|_{L^2 L^2}^2
&\leq C\|\nabla\phi_\e\|_{L^2L^2}\|\nabla\mu_\e\|_{L^2L^2}+  \int_0^T\int_{ \Omega_\e}  |F'(\phi_\e)|^2\dd x\dd t\\
&\leq C\|\nabla\phi_\e\|_{L^2L^2}\|\nabla\mu_\e\|_{L^2L^2}+  C\int_0^T\int_{ \Omega_\e}  (|\phi_\e|^6+1)\dd x\dd t\\
&\leq C\|\nabla\phi_\e\|_{L^2L^2}\|\nabla\mu_\e\|_{L^2L^2}+  C\|\phi_\e\|_{L^\infty H^1}^6+C\\
&\leq C.
\ea
Together with \eqref{uni-bound}, we can deduce
\be\label{mu e L2H1}
\|\mu_\e\|_{L^2 H^1}\leq C.
\ee

Combining \eqref{uni-bound}--\eqref{mu e L2H1}, we obtain
\ba\nn
&\|\vu_\e\|_{L^\infty L^2}+\|D(\vu_\e)\|_{L^2 L^2}\leq C \sqrt{\lambda_\e},\\
&\|\phi_\e\|_{L^\infty H^1}+\|\mu_\e\|_{L^2 H^1}+\int_{\Omega_\e} F(\phi_\e(t)) \dd x \leq C.
\ea
Exploiting the properties of zero extension and extension operator in Lemma \ref{extension lemma}, we deduce
\ba\label{uni-bound-tilde}
&\|\widetilde\vu_\e\|_{L^\infty L^2}+\|D(\widetilde\vu_\e)\|_{L^2 L^2}\leq C \sqrt{\lambda_\e},\\
&\|\widetilde\phi_\e\|_{L^\infty H^1}+\|\widetilde\mu_\e\|_{L^2 H^1}+\int_{\Omega} F(\widetilde\phi_\e(t)) \dd x \leq C.
\ea
Up to a subsequence, the following weak convergences hold:
\ba \label{Convergence of u}
&\widetilde\vu_\e \to \vu \ \mbox{weakly-(*) in}\ L^\infty(0,T; L^2(\Omega; \R^3)) \ \mbox{and weakly in}\ L^2(0,T; H^1_0(\Omega;\R^3)),\\
&\widetilde\phi_\e \to \phi \ \mbox{weakly-(*) in}\ L^\infty(0,T; H^1(\Omega)),\\
&\widetilde\mu_\e \to \mu \ \mbox{weakly in}\ L^2(0,T; H^1(\Omega)).
\ea
Furthermore, the divergence free condition $\dive \uu=0$ follows from $\dive \widetilde\uu_\e=0$.

\section{Equations in homogenized domain}\label{Equation in homogenized domain}
In this section, we will find equations satisfied by $(\widetilde\vu_\e,\widetilde\phi_\e,\widetilde\mu_\e)$ in homogenized domain.
\subsection{Momentum equations in homogenized domain}\label{Momentum equations in homogenized domains}
To find the limit momentum equations in $(0,T)\times\Omega$, we shall deal with divergence-free test function $\bm\varphi \in C_c^\infty([0,T)\times \Omega; \mathbb{R}^3)$, which is not supported in $[0,T)\times\Omega_\e$ generally. To take advantage of weak formulation \eqref{def-weak-momentum}, we decompose $\bm\varphi$ into
\be\nn
\bm\varphi = {g_\varepsilon }\bm\varphi + (1-{g_\varepsilon})\bm\varphi,
\ee
where $\{ g_\varepsilon\} _{\varepsilon>0}$ is a family of functions that vanish on the holes $\Omega\backslash \Omega_\e$ and converge to $1$ strongly in $W^{1,q}(\Omega)$, for some $q>2$. 
The integral related to $(1-g_\varepsilon)\bm\varphi$ can be controlled uniformly. Meanwhile, we have $g_\e \bvp\in C_c^\infty ([0,T)\times \Omega; \mathbb{R}^3)$. However, $\dive (g_{\e} \bm\varphi) \neq 0$ means that $g_{\e} \bm\varphi$ fails to be an admissible test function for \eqref{def-weak-momentum}. To overcome this technical gap, we employ a Bogovskii type operator defined on $\Omega_{\e}$, see Lemma \ref{lem-div}. 
We obtain the following proposition. 
\begin{proposition}\label{moment-equa}
Under the assumptions in Theorem \ref{Main theorem}, 
there exist $\GG_\varepsilon\in {\cal D}'((0,T)\times\Omega;\mathbb{R}^3)$, such that for any ${\bm\varphi}\in C_c^\infty([0,T)\times \Omega;\mathbb{R}^3)$ with $\dive {\bm\varphi}=0$, there holds
\ba\label{momentum eq-1}
&-\int_{0}^T\int_{\Omega}   \widetilde\uu_\e \cdot \pa_t \bm\varphi \dd x \dd t
-\int_0^T\int_{\Omega}  \widetilde\uu_\e\otimes\widetilde\uu_\e:\nabla \bm\varphi \dd x \dd t+
\int_0^T\int_{\Omega} \nu (\widetilde\phi_\e) D(\widetilde\uu_\e):D(\bm\varphi)\dd x \dd t\\
&\quad+\int_0^T\int_{\Omega} \lambda_\e \widetilde\phi_\e \nabla \widetilde\mu_\e\cdot \bm\varphi\dd x \dd t
=\int_0^T\int_{\Omega} \mathbf g_\e\cdot \bm\varphi \dd x \dd t
+\int_{\Omega}\widetilde\uu_\e^0 \cdot  \bm\varphi(0) \dd x
+\l\GG_\varepsilon, \bvp \r,
\ea
with
\begin{equation}\label{est-GG}
\left|\l\GG_\varepsilon, \bvp \r \right| \leq C\e^{\sigma} \lambda_\e^{1\over2}(\|\pa_t \bvp\|_{L^{4\over3} L^2}+ \|\nabla\bvp\|_{L^4 L^{r_1}}+ \|\bvp(0)\|_{L^{r_2}}).
\end{equation}
Here, $\sigma=\frac{(3-q)\alpha-3}{q}>0$ for some $q\in(2,3)$ close to $2$, $r_1\in (2,3)$ and $r_2>6$ are given in \eqref{q-r-1-r-2}.
\end{proposition}

\begin{proof} 
Let ${\bm\varphi}\in C_c^\infty([0,T)\times \Omega;\mathbb{R}^3)$ with $\dive {\bm\varphi}=0$. Combining the distribution of holes in \eqref{defi-holes-1}--\eqref{defi-holes-2}, there exist cut-off functions $\{ g_\varepsilon \}_{\varepsilon>0} \subset C^\infty(\R^{3})$ such that $0\leq g_{\e} \leq 1$ and
\begin{equation}\label{defi-g}
 g_\varepsilon= 0 \ \mbox{on} \ \bigcup_{k \in K_\varepsilon} B(x_{\e,k},\delta_0 \varepsilon^\alpha),\quad g_\varepsilon = 1 \ \mbox{on} \ (\bigcup_{k \in K_\varepsilon} B(x_{\e,k},\delta_1 \varepsilon^\alpha))^c,\quad |\nabla g_{\e}| \leq C \e^{-\alpha}.
\end{equation}
Then for each $1\leq q \leq \infty$, there hold
\begin{equation}\label{est-g}
\| g_\varepsilon- 1\| _{L^q(\R^{3})} \leq C \varepsilon^{\frac{3\alpha-3}{q}},\quad
\| \nabla g_\varepsilon\| _{L^q(\R^{3})} \leq C\varepsilon ^{\frac{3\alpha-3}{q} -\alpha}.
\end{equation}

Since ${\bm\varphi}$ is not supported in $[0,T)\times \Omega_\e$, to find a proper test function in weak formulation \eqref{def-weak-momentum}, we firstly decompose $\bvp$ into $\bvp = {g_\varepsilon }\bvp + (1-{g_\varepsilon})\bvp$, then we have
\ba\nn
\l\GG_\varepsilon, \bvp \r
=&-\int_{0}^T \int_{\Omega} \widetilde\uu_\e \cdot \pa_t(g_\e\bm\varphi) \dd x \dd t
-\int_0^T\int_{\Omega}  \widetilde\uu_\e\otimes\widetilde\uu_\e:\nabla (g_\e\bm\varphi) \dd x \dd t\\
&\quad+\int_0^T\int_{\Omega} \nu (\widetilde\phi_\e) D(\widetilde\uu_\e):D(g_\e\bm\varphi)\dd x \dd t
+\int_0^T\int_{\Omega} \lambda_\e \widetilde\phi_\e \nabla \widetilde\mu_\e\cdot (g_\e\bm\varphi) \dd x \dd t\\
&\quad-\int_0^T\int_{\Omega} \mathbf g_\e\cdot (g_\e\bm\varphi) \dd x \dd t
-\int_{\Omega} \vu_\e^0 \cdot g_\e\bm\varphi(0) \dd x
+\sum_{i=1}^{6} {I_i},
\ea
with
\ba\nn
I_1 & = -\int_{0}^T \int_{\Omega} \widetilde\uu_\e \cdot\pa_t( (1-g_\e)\bm\varphi) \dd x \dd t, 
&&I_2  = -\int_0^T\int_{\Omega}  \widetilde\uu_\e\otimes\widetilde\uu_\e:\nabla ((1-g_\e)\bm\varphi) \dd x \dd t ,\\
I_3 &= \int_0^T\int_{\Omega} \nu (\widetilde\phi_\e) D(\widetilde\uu_\e):D((1-g_\e)\bm\varphi)\dd x \dd t ,
&&I_4 =\int_0^T\int_{\Omega} \lambda_\e \widetilde\phi_\e \nabla \widetilde\mu_\e\cdot ((1-g_\e)\bm\varphi) \dd x \dd t,\\
I_5 &=-\int_0^T\int_{\Omega} \mathbf g_\e\cdot ((1-g_\e)\bm\varphi) \dd x \dd t,
&&I_6 =-\int_{\Omega} \vu_\e^0 \cdot(1-g_\e)\bm\varphi(0) \dd x.
\ea
By virtue of \eqref{defi-g}, we can derive
\ba\nn
\int_{\Omega_\e} \dive ( g_{\e}\bvp) \dx = \int_{\Omega_{\e}} \bvp \cdot \nabla g_{\e}\dx=\int_{\Omega} \bvp \cdot \nabla g_{\e}\dx = \int_{\Omega} \dive ( g_{\e}\bvp) \dx = 0,
\ea
which yields $\bvp \cdot \nabla  g_{\e}\in L^{\tilde q}_0(\Omega_\e)$ for $\tilde q>1$ and $t\in (0,T)$. 
Since $g_\e\bvp$ is not divergence-free, we apply Lemma \ref{lem-div} and introduce
\ba\nn
\bvp_{1} := g_\e\bvp - \bvp_{2},\quad \bvp_{2} := \calB_{\e} (\dive(\bvp g_{\e})) = \calB_{\e} (\bvp \cdot \nabla g_\e),
\ea
then a direct computation yields $\dive {\bm\varphi_1}=0$. By virtue of weak formulation \eqref{def-weak-momentum}, we deduce
\begin{align*}
\l\GG_\varepsilon, \bvp \r=
&-\int_{0}^T \int_{\Omega_\e}   \uu_\e \pa_t\bm\varphi_1 \dd x \dd t
 -\int_0^T\int_{\Omega_\e}  \uu_\e\otimes \uu_\e:\nabla \bm\varphi_1 \dd x \dd t+
 \int_0^T\int_{\Omega_\e} \nu (\phi_\e) D(\uu_\e):D(\bm\varphi_1)\dd x \dd t\\
&
\quad+ \int_0^T\int_{\Omega_\e} \lambda_\e \phi_\e \nabla \mu_\e\cdot \bm\varphi_1 \dd x \dd t
- \int_0^T\int_{\Omega_\e} \mathbf g_\e\cdot \bm\varphi_1 \dd x \dd t
-\int_{\Omega_\e} \vu_\e^0 \cdot \bm\varphi_1(0) \dd x
+\sum_{i=1}^{6} {I_i}+\sum_{ i=7}^{12} {I_i}\\
=&\sum_{ i=1}^{6} {I_i} +\sum_{ i=7}^{12} {I_i},
\end{align*}
with
\begin{align*}
I_7 &=-\int_{0}^T \int_{\Omega_\e}  \uu_\e \cdot\pa_t\bm\varphi_2 \dd x \dd t,
&&I_8 =  -\int_0^T\int_{\Omega_\e} \widetilde\uu_\e\otimes\widetilde\uu_\e:\nabla \bm\varphi_2 \dd x \dd t, \\
I_9 &= \int_0^T\int_{\Omega_\e} \nu (\phi_\e) D(\uu_\e):D(\bm\varphi_2)\dd x \dd t,
&&I_{10} = \int_0^T\int_{\Omega_\e} \lambda_\e \phi_\e \nabla \mu_\e\cdot \bm\varphi_2 \dd x \dd t,\\
I_{11} &=- \int_0^T\int_{\Omega_\e}\mathbf g_\e\cdot \bm\varphi_2 \dd x \dd t,
&&I_{12} =-\int_{\Omega_\e} \vu_\e^0 \cdot \bm\varphi_2(0) \dd x.
\end{align*}

To derive uniform bounds on $\l\GG_\varepsilon, \bvp \r$, we proceed to establish uniform estimates for each integral $I_{i} \ (1\leq i\leq 12)$.  

Since $\alpha>3$, there exists $q\in (2,3)$ close to $2$, such that
\be\nn
\sigma:=\frac{(3-q)\alpha-3}{q}>0.
\ee 
By \eqref{uni-bound-tilde} and \eqref{est-g}, together with interpolation inequality and Sobolev embedding theorem, we deduce
\ba\nn
|I_1| &\leq C \| \widetilde \uu_\e\|_{L^4 L^3}\|1-g_\e\|_{L^6}  \|\pa_t\bm\varphi\|_{L^{4\over3} L^2}\\
&\leq C\| \widetilde \uu_\e\|_{L^\infty L^2}^{1\over2} \| \widetilde \uu_\e\|_{L^2 L^6}^{1\over2}  \|1-g_\e\|_{L^6}  \|\pa_t\bm\varphi\|_{L^{4\over3} L^2}\\
&\leq C \lambda_\e^{1\over2}  \e^{\frac{\alpha-1}{2}} \|\pa_t\bm\varphi\|_{L^{4\over3} L^2}
\ea
We use interpolation inequality and Sobolev embedding again to obtain
\be\nn
|I_2| \leq C \|\widetilde \uu_\e\|_{L^2 L^6} \|\widetilde \uu_\e\|_{L^4 L^3}
(\|1-g_\e\|_{L^{q^*}}\|\nabla\bm\varphi\|_{L^4 L^{r_1}}+ \|\nabla g_\e\|_{L^q}\|\bm\varphi\|_{L^4 L^{r_2}})
\leq C\lambda_\e \e^{\sigma}\|\nabla\bm\varphi\|_{L^4 L^{r_1}},
\ee
where $q^*$, $r_1$ and $r_2$ satisfy
\be\label{q-r-1-r-2}
\frac{1}{q^*}=\frac{1}{q}-\frac{1}{3}<\frac{1}{6},\quad  \frac{1}{r_1}=\frac{1}{2}-\frac{1}{q^*}>\frac{1}{3},
\quad \frac{1}{r_2}=\frac{1}{2}-\frac{1}{q}=\frac{1}{6}-\frac{1}{q^*}=\frac{1}{r_1}-\frac{1}{3}.
\ee
Similarly, by \eqref{assumption-g}, \eqref{uni-bound-tilde} and \eqref{est-g}, we can deduce
\ba\nn
|I_3| &\leq C\nu_2 \|D(\widetilde \vu_\e)\|_{L^2 L^2}\|D((1-g_\e)\bm\varphi)\|_{L^2 L^2}\\
&\leq C\|D(\widetilde \vu_\e)\|_{L^2 L^2}( \|1-g_\e\|_{L^{q^*}} \|\nabla\bm\varphi\|_{L^4 L^{r_1}}+\|\nabla g_\e\|_{L^q}\|\bm\varphi\|_{L^4 L^{r_2}})\\
&\leq C\lambda_\e^{1\over2} \e^{\sigma}\|\nabla\bm\varphi\|_{L^4 L^{r_1}},
\ea
\ba\nn
|I_4| \leq C \lambda_\e \|\widetilde\phi_\e\|_{L^\infty L^6} \|\nabla\widetilde\mu_\e\|_{L^2 L^2} \|1-g_\e\|_{L^{q^*}}\|\bm\varphi\|_{L^4 L^{r_2}}
\leq C\lambda_\e \e^{\sigma}\|\nabla\bm\varphi\|_{L^4 L^{r_1}},
\ea
\ba\nn
|I_5| \leq C\|\mathbf g_\e \|_{L^2 L^2} \|1-g_\e\|_{L^{q^*}}\|\bm\varphi\|_{L^4 L^{r_2}}
\leq C\lambda_\e^{\frac{1}{2}} \e^{\sigma}\|\nabla\bm\varphi\|_{L^4 L^{r_1}},
\ea
and
\ba\nn
|I_6|\leq C\|\vu_\e^0\|_{L^2} \|1-g_\e\|_{L^q} \|\bvp(0)\|_{L^{r_2}}
\leq C \lambda_\e^{1\over2}  \e^{\sigma}\|\bvp(0)\|_{L^{r_2}}.
\ea

Next, we will use the properties of the Bogovskii operator $\calB_{\e}$ in Lemma \ref{lem-div} to estimate $I_{i}$ $(7\leq i \leq12)$. Notice that
\be\nn
 \int_{\Omega_{\e}} \pa_t\bvp \cdot \nabla g_{\e}\dx=\int_{\Omega}\pa_t \bvp \cdot \nabla g_{\e}\dx = \int_{\Omega} \dive ( g_{\e}\pa_t\bvp) \dx = 0,
\ee
and the Bogovskii operator $\calB_{\e}$ only acts on spatial variable, there holds
 \be\nn
 \pa_t \bvp_2=\pa_t \calB_{\e} (\bvp \cdot \nabla g_\e)=\calB_{\e} (\pa_t\bvp \cdot \nabla g_\e).
 \ee
Since $q>2$ is close to $2$, we can choose $r_3>\frac{3}{2}$ close to $\frac{3}{2}$ and $r_4>1$ close to $1$ such that
\be\nn
\frac{1}{r_3}=\frac{1}{r_4}-\frac{1}{3}, \quad \frac{1}{r_4}=\frac{1}{q}+\frac{1}{2}.
\ee
Utilizing \eqref{Bef}, \eqref{uni-bound-tilde}, \eqref{est-g} and Sobolev embedding, we deduce
 \ba\nn
|I_7| &\leq C\| \vu_\e\|_{L^4 L^3} \|\pa_t\bm\varphi_2\|_{L^{4\over 3} L^{r_3}} 
= C\|\vu_\e\|_{L^4 L^3} \|\calB_{\e} (\pa_t\bvp \cdot \nabla g_\e)\|_{L^{4\over 3} L^{r_3}}  \\
&\leq C\|\vu_\e\|_{L^4 L^3} \|\calB_{\e} (\pa_t\bvp \cdot \nabla g_\e)\|_{L^{4\over 3} W^{1,r_4}} 
\leq C\|\vu_\e\|_{L^4 L^3} \|\pa_t\bvp \cdot \nabla g_\e\|_{L^{4\over 3} L^{r_4}} \\
&\leq C\|\vu_\e\|_{L^4 L^3} \|\pa_t\bvp \|_{L^{4\over 3} L^{2}} \|\nabla g_\e\|_{L^q}
\leq C\lambda_\e^{1\over2} \e^\sigma\|\pa_t\bvp \|_{L^{4\over 3} L^{2}}.
\ea
By \eqref{uni-bound}, similar arguments give
\ba\nn
|I_8|&  \leq C\|\vu_\e\|_{L^2 L^6} \|\vu_\e\|_{L^4 L^3} \|\nabla\bm\varphi_2\|_{L^4 L^2}
\leq C \lambda_\e \|\calB_{\e} (\bvp \cdot \nabla g_\e)\|_{L^4 H^1}
\leq C \lambda_\e \|\bvp \cdot \nabla g_\e\|_{L^4 L^{2}}\\
&\leq C \lambda_\e \|\bvp \|_{L^4 L^{r_2}}\| \nabla g_\e\|_{L^q}
\leq C\lambda_\e \e^{\sigma}\|\nabla\bm\varphi\|_{L^4 L^{r_1}},
\ea
\ba\nn
|I_9|  \leq C\nu_2 \|D(\vu_\e)\|_{L^2 L^2} \|D(\bvp_2)\|_{L^2 L^2}
\leq C\|D(\vu_\e)\|_{L^2 L^2} \|\calB_{\e} (\bvp \cdot \nabla g_\e)\|_{L^2 H^1}
\leq C\lambda_\e^{1\over2} \e^{\sigma}\|\nabla\bm\varphi\|_{L^4 L^{r_1}},
\ea
and
\ba\nn
|I_{10}|&  \leq C\lambda_\e \|\phi_\e\|_{L^\infty L^6} \|\nabla\mu_\e\|_{L^2 L^2} \|\bm\varphi_2\|_{L^2 L^6}
\leq C\lambda_\e \|\calB_{\e} (\bvp \cdot \nabla g_\e)\|_{L^2 H^1}
\leq C\lambda_\e \e^{\sigma}\|\nabla\bm\varphi\|_{L^4 L^{r_1}}.
\ea
Furthermore, by virtue of \eqref{assumption-g} and \eqref{est-g}, we obtain 
\ba\nn
|I_{11}|&\leq C\|\mathbf g_\e \|_{L^2 L^2} \|\bm\varphi_2\|_{L^2 L^2}
\leq C\|\mathbf g_\e \|_{L^2 L^2} \|\calB_{\e} (\bvp \cdot \nabla g_\e)\|_{L^2 H^1}
\leq C\lambda_\e^{1\over2} \e^{\sigma}\|\nabla\bm\varphi\|_{L^4 L^{r_1}},
\ea
and
\ba\nn
|I_{12}|\leq C\|\vu_\e^0\|_{L^2}\|\bm\varphi_2(0)\|_{L^2}\leq C\|\vu_\e^0\|_{L^2}\|\bvp(0)\|_{L^{r_2}}\|\nabla g_\e\|_{L^q}\leq C\lambda_\e^{1\over2} \e^{\sigma}\|\bm\varphi(0)\|_{L^{r_2}},
\ea

Summing up the estimates for $I_{i}$ $(i=1,\cdots,10)$ above, we derive
\ba\nn
|\l\GG_\varepsilon, \bvp \r| \leq C \e^{\sigma}(\lambda_\e+\sqrt{\lambda_\e})(\|\pa_t\bvp\|_{L^{4\over3} L^2}+\|\nabla\bvp\|_{L^4 L^{r_1}}+\|\bm\varphi(0)\|_{L^{r_2}}).
\ea
Since $\lambda_\e\leq C \sqrt{\lambda_\e}$, no matter $\lambda_\e\to\lambda> 0$ or $\lambda_\e\to\lambda= 0$,
this yields \eqref{est-GG}.
\end{proof}

\subsection{Phase-field equation in homogenized domain}\label{Phase-field evolution equation in homogenized domain}
In this subsection, we derive the limit phase-field equation in $(0,T)\times\Omega$. For this purpose, we analyze the weak formulation of phase-field equation satisfied by the extended variables $( \widetilde\vu_\e, \widetilde\phi_\e, \widetilde\mu_\e)$ and obtain the following proposition.
\begin{proposition}\label{phase-field-equa}
Under the assumptions in Theorem \ref{Main theorem}, 
there exist $\HH_\varepsilon\in {\cal D}'((0,T)\times\Omega)$, 
such that for any $\varphi\in C_c^\infty([0,T)\times \Omega)$, there holds
\ba\nn
&-\int_0^T\int_{\Omega}  \widetilde\phi_\e\pa_t\varphi \dd x\dd t+
\int_0^T\int_{\Omega} M(\widetilde\phi_\e) \nabla\widetilde\mu_\e\cdot\nabla \varphi \dd x \dd t\\
&\quad=\int_0^T\int_{\Omega}  \widetilde\phi_\e \widetilde\uu_\e\cdot\nabla \varphi \dd x \dd t
+\int_{\Omega}  \phi_\e^0\varphi(0) \dd x+
\l\HH_\varepsilon, \varphi \r,
\ea
with
\begin{equation}\label{est-H}
\left|\l\HH_\varepsilon, \varphi \r\right| \leq C\e^{\sigma}(\|\pa_t\varphi\|_{L^2 L^{3\over2}}+\|\nabla\varphi\|_{L^4 L^{r_1}}+ \|\varphi(0)\|_{L^{3}}).
\end{equation}
Here, $\sigma>0$ and $r_1\in (2,3)$ are the same as in Proposition \ref{moment-equa}.
\end{proposition}
\begin{proof}
Let $\varphi\in C_c^\infty([0,T)\times \Omega)$. 
Exploiting $g_\e$ given in \eqref{defi-g}, we decompose $\varphi$ into 
$$\varphi  = {g_\varepsilon }\varphi + (1-{g_\varepsilon})\varphi,$$ 
then $\supp({g_\varepsilon }\varphi)\subset [0,T)\times\Omega_\e$ gives
\ba\label{Hevarphi}
\l\HH_\varepsilon, \varphi \r =
&-\int_0^T\int_{\Omega_\e}  \phi_\e\pa_t (g_\e\varphi) \dd x\dd t
 +\int_0^T\int_{\Omega_\e}  M(\phi_\e)\nabla\mu_\e\cdot\nabla (g_\e\varphi) \dd x \dd t\\
&\quad
- \int_0^T\int_{\Omega_\e}  \phi_\e \uu_\e\cdot\nabla (g_\e\varphi) \dd x \dd t
-\int_{\Omega_\e}  \phi_\e^0 g_\e\varphi(0) \dd x
+\sum_{i=1}^{4} {J_i}\\
=&\sum_{i=1}^{4} {J_i},
\ea
with
\ba\nn
J_1 & =- \int_0^T \int_{\Omega}\widetilde\phi_\e\pa_t((1-g_\e)\varphi) \dd x\dd t, 
&&J_2  =\int_0^T\int_{\Omega}  M(\widetilde\phi_\e)\nabla\widetilde\mu_\e\cdot\nabla ((1-g_\e)\varphi)\dd x \dd t,\\
J_3 & =-\int_0^T\int_{\Omega} \widetilde\phi_\e \widetilde\uu_\e\cdot\nabla ((1-g_\e)\varphi) \dd x \dd t,
&&J_4  =-\int_{\Omega}  \phi_\e^0 (1-g_\e)\varphi(0) \dd x.
\ea
We proceed to estimate each integral $J_{i}  \ (1\leq i\leq 4)$. By virtue of \eqref{uni-bound-tilde} and \eqref{est-g}, one has
\ba\nn
|J_1| \leq C \|\widetilde\phi_\e\|_{L^\infty L^6} \|1-g_\e\|_{L^{6}}      \|\pa_t\varphi\|_{L^2 L^{3\over2}}
\leq C\e^\sigma  \|\pa_t\varphi\|_{L^2 L^{3\over2}}.
\ea
Similarly, using \eqref{initial-u-phi}, \eqref{uni-bound-tilde}, \eqref{est-g} and Lemma \ref{extension lemma}, we deduce
\ba\label{J23}
&|J_2| \leq CM_2\|\nabla \widetilde \mu_\e\|_{L^2 L^2} ( \|1-g_\e\|_{L^{q^*}} \|\nabla\varphi\|_{L^4 L^{r_1}}+\|\nabla g_\e\|_{L^q}\|\varphi\|_{L^4 L^{r_2}})
\leq C \e^{\sigma}\|\nabla\varphi\|_{L^4 L^{r_1}},\\
&|J_3| \leq C\| \widetilde \phi_\e\|_{L^\infty L^6} \| \widetilde \uu_\e\|_{L^2 L^6} ( \|1-g_\e\|_{L^{q^*}} \|\nabla\varphi\|_{L^4 L^{r_1}}+\|\nabla g_\e\|_{L^q}\|\varphi\|_{L^4 L^{r_2}})
\leq C\lambda_\e^{1\over2} \e^{\sigma}\|\nabla\varphi\|_{L^4 L^{r_1}},
\ea
and
\ba\nn
|J_4| \leq C \|\phi_\e^0\|_{L^2}\|1-g_\e\|_{L^{6}}      \|\varphi(0)\|_{L^{3}}
\leq C\e^\sigma  \|\varphi(0)\|_{L^3}.
\ea
Combining all above estimates, we arrive at the desired estimate \eqref{est-H}.
\end{proof}

\subsection{Chemical-potential equation in homogenized domain}\label{Chemical-potential equation in homogenized domain}

Our next goal is to derive the homogenized equation governing the limit chemical potential in
 $(0,T)\times \Omega$. To this end, we study the weak formulation of chemical-potential equation satisfied by the extended variables $(\widetilde\mu_\e,\widetilde\phi_\e)$ and give the following proposition.
\begin{proposition}\label{chemical-equa}
Under the assumptions in Theorem \ref{Main theorem}, there exist $\KK_\varepsilon \in{\cal D}'((0,T)\times\Omega)$, such that for any $\psi\in C_c^\infty((0,\infty)\times\Omega)$, there holds
\ba\nn
\int_0^T\int_{ \Omega}  \widetilde\mu_\e \psi \dd x \dd t
-\int_0^T\int_{ \Omega} \nabla\widetilde\phi_\e \cdot \nabla\psi \dd x\dd t
=\int_0^T\int_{ \Omega}  F'(\widetilde\phi_\e) \psi \dd x\dd t
+\l\KK_\varepsilon, \psi \r,
\ea
with 
\begin{equation}\label{est-K}
|\l\KK_\varepsilon, \psi \r| \leq C \e^{\sigma}\|\nabla\psi\|_{L^2L^{r_1}}.
\end{equation}
Here, $\sigma>0$ and $r_1\in (2,3)$ are the same as in Proposition \ref{moment-equa}.
\end{proposition}

\begin{proof}
For $\psi\in C_c^\infty((0,\infty)\times\Omega)$, to make use of \eqref{def-weak-chemical}, we
utilize $g_\e$ given in \eqref{defi-g} to decompose $\psi$ into 
$$\psi  = {g_\varepsilon }\psi + (1-{g_\varepsilon})\psi,$$ 
then we have
\ba\nn
\l\KK_\varepsilon, \psi \r
&=\int_0^T\int_{ \Omega_\e}  \mu_\e (g_\e\psi) \dd x \dd t
-\int_0^T\int_{ \Omega_\e} \nabla\phi_\e \cdot \nabla (g_\e\psi) \dd x\dd t
-\int_0^T\int_{ \Omega_\e}  F'(\phi_\e) (g_\e\psi) \dd x\dd t+\sum_{i=1}^{3} {K_i}\\
&=\sum_{i=1}^{3} {K_i},
\ea
with
\ba\nn
K_1 & =\int_0^T\int_{ \Omega}  \widetilde\mu_\e (1-g_\e)\psi \dd x\dd t, \\
K_2 & =-\int_0^T\int_{ \Omega} \nabla\widetilde\phi_\e \cdot \nabla((1-g_\e)\psi) \dd x\dd t,\\
K_3 & =-\int_0^T\int_{ \Omega}  F'(\widetilde\phi_\e) (1-g_\e)\psi \dd x\dd t .
\ea
We proceed to establish uniform estimates for each integral $K_{i} \ (1\leq i\leq 3)$. 
By \eqref{uni-bound-tilde} and \eqref{est-g}, for $q^*$, $r_1$ and $r_2$ given in \eqref{q-r-1-r-2}, there holds
\ba\nn
|K_1|\leq C \|\widetilde\mu_\e\|_{L^2L^6} \|1-g_\e\|_{L^2} \|\psi\|_{L^2L^{r_2}}\leq C \e^{\frac{3\alpha-3}{2}}\|\psi\|_{L^2L^{r_2}},
\ea
\ba\nn
|K_2|\leq& C \|\nabla\widetilde\phi_\e\|_{L^2L^2} \|\nabla((1-g_\e)\psi)\|_{L^2L^2}\\
\leq& C \|\nabla\widetilde\phi_\e\|_{L^2L^2} (\|\nabla g_\e\|_{L^q}\|\psi\|_{L^2 L^{r_2}}+\|1- g_\e\|_{L^{q^*}}\|\nabla\psi\|_{L^2L^{r_1}})\\
\leq& C \e^{\sigma}\|\nabla\psi\|_{L^2L^{r_1}},
\ea
and
\ba\nn
|K_3|\leq C (\|\widetilde\phi_\e\|_{L^\infty L^6}^3+1)\|1- g_\e\|_{L^{q^*}}\|\psi\|_{L^2L^{r_1}}\leq C \e^{\sigma}\|\psi\|_{L^2L^{r_1}}.
\ea
Combining all above estimates, we arrive at the desired estimate.
\end{proof}

\section{Asymptotic limit}\label{Asymptotic limit}
In this section, we derive the limit equations for system \eqref{NSCH}--\eqref{boundary condition} in homogenized domain. By virtue of \eqref{est-GG} and assumptions \eqref{initial-u-phi}--\eqref{assumption-g}, we are able to pass the terms on the right-hand side of \eqref{momentum eq-1} to the limit as $\e\to 0$. However, the left-hand side of \eqref{momentum eq-1} contains a nonlinear convective term. To pass to the limit, we first investigate the convergence of this nonlinear convective term.

\subsection{Convergence of the nonlinear convective term}\label{Convergence of the nonlinear convective term}
Since the convective term $\dive(\widetilde\uu_\e\otimes\widetilde\uu_\e)$ 
is nonlinear, weak convergence of $\widetilde\vu_\e$ is not sufficient to pass these products to the homogenized limit. To overcome this difficulty, we employ Aubin--Lions type compactness arguments, see Lemma \ref{lem-AL}.

 By Proposition \ref{moment-equa},
for any ${\bm\varphi}\in C_c^\infty((0,T)\times \Omega;\mathbb{R}^3)$ with $\dive {\bm\varphi}=0$, there holds
\ba\nn
 \langle\pa_t \widetilde\uu_\e , \bm\varphi\rangle 
=&
\int_0^T\int_{\Omega}  \widetilde\uu_\e\otimes\widetilde\uu_\e:\nabla \bm\varphi \dd x \dd t
-\int_0^T\int_{\Omega} \nu (\widetilde\phi_\e) D(\widetilde\uu_\e):D(\bm\varphi)\dd x \dd t\\
&\quad-\int_0^T\int_{\Omega} \lambda_\e \widetilde\phi_\e \nabla \widetilde\mu_\e\cdot \bm\varphi\dd x \dd t
+\int_0^T\int_{\Omega} \mathbf g_\e\cdot \bm\varphi \dd x \dd t
+\l\GG_\varepsilon, \bvp \r,
\ea
with
\begin{equation}\nn
\left|\l\GG_\varepsilon, \bvp \r \right| \leq C\e^{\sigma} \lambda_\e^{1\over2}(\|\pa_t \bvp\|_{L^{4\over3} L^2}+ \|\nabla\bvp\|_{L^4 L^{r_1}}).
\end{equation}
Noticing that
\ba\nn
&\left|\int_0^T\int_{\Omega}  \widetilde\uu_\e\otimes\widetilde\uu_\e:\nabla \bm\varphi \dd x \dd t \right|
\leq \|\widetilde\vu_\e\|_{L^4 L^3} \|\widetilde\vu_\e\|_{L^2 L^6}\| \nabla\bm \varphi\|_{L^4 L^2}
\leq C \|\widetilde\vu_\e\|_{L^\infty L^2}^{1\over2} \|\widetilde\vu_\e\|_{L^2 H^1}^{3\over2} \|\bm \varphi\|_{L^4 H^1},\\
&\left| \int_0^T\int_{\Omega_\e} \nu (\widetilde\phi_\e) D(\widetilde\uu_\e):D(\bm\varphi)\dd x \dd t\right|
\leq \nu_2 \|D(\widetilde\vu_\e)\|_{L^2 L^2}\|D(\bm \varphi)\|_{L^2 L^2}
\leq C \|\widetilde\vu_\e\|_{L^2 H^1} \|\bm \varphi\|_{L^4 H^1},\\
&\left| \int_0^T\int_{\Omega_\e} \lambda_\e \widetilde\phi_\e \nabla \widetilde\mu_\e\cdot \bm\varphi\dd x \dd t\right| 
\leq  \lambda_\e\|\widetilde\phi_\e\|_{L^\infty L^6} \|\nabla \widetilde\mu_\e\|_{L^2 L^2} \|\bm \varphi\|_{L^2 L^3}
\leq C \lambda_\e\|\widetilde\phi_\e\|_{L^\infty H^1} \|\nabla \widetilde\mu_\e\|_{L^2 L^2} \|\bm \varphi\|_{L^4 H^1},\\
&\left| \int_0^T\int_{\Omega_\e} \mathbf g_\e\cdot \bm\varphi \dd x \dd t\right| 
\leq \|\mathbf g_\e\|_{L^2 L^2}\|\bm\varphi\|_{L^2 L^2}
\leq C\sqrt{\lambda_\e}  \|\mathbf g_0\|_{L^2 L^2} \|\bm \varphi\|_{L^4 H^1},
\ea
together with \eqref{uni-bound-tilde}, we can deduce
\ba\label{pat u varphi}
|\langle\pa_t \widetilde\uu_\e , \bm\varphi\rangle |\leq C\lambda_\e^{1\over2}(\|\bvp\|_{L^4 W^{1,r_1}}+\e^\sigma \|\pa_t \bvp\|_{L^{4\over3} L^2}),
\ea
where $\sigma>0$ and $r_1\in(2,3)$ are defined in Proposition \ref{moment-equa}. 
Estimate \eqref{pat u varphi} holds for all ${\bm\varphi}\in C_c^\infty((0,T)\times \Omega;\mathbb{R}^3)$ since $ \langle\pa_t \widetilde\uu_\e , \bm\varphi\rangle= \langle\pa_t \widetilde\uu_\e ,\mathbb{P} \bm\varphi\rangle$, where $\mathbb{P}$ is the Leray--Helmholtz projection operator. Thus we have the following decomposition
\ba\nn
\widetilde\vu_\e=\widetilde\vu_\e^{(1)}+\e^\sigma\widetilde\vu_\e^{(2)},
\ea
where $\pa_t \widetilde\vu_\e^{(1)}$ is uniformly bounded in $L^{4\over3}(0,T; W^{-1,r_1'}(\Omega;\mathbb{R}^3))$, and $\widetilde\vu_\e^{(2)}$ is uniformly bounded in $L^4(0,T;L^2(\Omega;\mathbb{R}^3))$.
By virtue of \eqref{uni-bound-tilde} and interpolation inequality, we deduce that $\widetilde\vu_\e$ is uniformly bounded in $L^4(0,T;L^3(\Omega;\mathbb{R}^3))$, and thus $\widetilde\vu_\e^{(1)}=\widetilde\vu_\e-\e^\sigma\widetilde\vu_\e^{(2)}$ is uniformly bounded in $L^4(0,T;L^2(\Omega;\mathbb{R}^3))$. 
Together with the fact that $\e^\sigma\widetilde\vu_\e^{(2)} \to 0$ strongly in $L^4(0,T;L^2(\Omega;\mathbb{R}^3))$,
up to a subsequence, we have
\be\nn
\widetilde\vu_\e^{(1)}\to \vu \text{ weakly in } L^4(0,T;L^2(\Omega;\mathbb{R}^3)).
\ee
Since $\widetilde\vu_\e$ is uniformly bounded in $L^2(0,T;H_0^{1}(\Omega;\mathbb{R}^3))$, there holds
\be\nn
\|\widetilde\vu_\e-\widetilde\vu_\e(\cdot+\xi,t)\|_{L^{2}L^{2}}\leq |\xi| \cdot \| \nabla\widetilde\vu_\e\|_{L^{2}L^{2}}\leq C|\xi| 
\to 0, \quad \text{as }|\xi|\to 0.
\ee
Applying Lemma \ref{lem-AL}, we can deduce
\be\nn
\widetilde\vu_\e^{(1)}\otimes\widetilde\vu_\e \to \vu\otimes\vu \text{ in } {\cal D'}((0,T)\times \Omega).
\ee
Since $\widetilde\vu_\e^{(2)}$ is uniformly bounded in $L^4(0,T;L^2(\Omega;\mathbb{R}^3))$, $\widetilde\vu_\e$ is uniformly bounded in $L^2(0,T;L^6(\Omega;\mathbb{R}^3))$ $\cap L^\infty(0,T;L^2(\Omega;\mathbb{R}^3))$,
direct computation yields
\be\nn
\e^\sigma\widetilde\vu_\e^{(2)}\otimes\widetilde\vu_\e \to {\bm 0}  \text{ strongly in }L^{4\over3}(0,T;L^{3\over2}(\Omega;\mathbb{R}^{3\times 3}))\cap L^{4}(0,T;L^{1}(\Omega;\mathbb{R}^{3\times 3})).
\ee
Thus we have
\be\nn
\widetilde\vu_\e\otimes\widetilde\vu_\e\to \vu\otimes\vu \text{ in } {\cal D'}((0,T)\times \Omega).
\ee
Combining the fact that $\widetilde\vu_\e$ is uniformly bounded in $L^4(0,T;L^3(\Omega;\mathbb{R}^3))$ and $L^2(0,T;L^6(\Omega;\mathbb{R}^3))$, we derive
\be\label{convective term}
\widetilde\vu_\e\otimes\widetilde\vu_\e\to \vu\otimes\vu \text{ weakly in } L^{4\over3}(0,T;L^{2}(\Omega;\mathbb{R}^{3\times 3})).
\ee

Taking identity matrix as a test function in \eqref{convective term} gives
\be\nn
\int_0^{T} \int_{\Omega} \widetilde\vu_\e\cdot\widetilde\vu_\e \dd x \dd t \to
\int_0^{T} \int_{\Omega} \vu\cdot\vu \dd x \dd t, 
\ee
which combined with \eqref{Convergence of u} gives
\be\label{strong convergence of u}
\widetilde\vu_\e \to \vu \text{ in } L^{2}(0,T;L^{2}(\Omega;\mathbb{R}^3)).
\ee

\subsection{Strong convergence of $\widetilde\phi_\e$}\label{Strong convergence of phi}
There are nonlinear terms involving $\widetilde\phi_\e$ in the left-hand side of the momentum equations \eqref{momentum eq-1} .
To deal with them, we will first show the strong convergence of $\widetilde\phi_\e$.

Following the proof of Proposition \ref{phase-field-equa}, by \eqref{Hevarphi}, for any $\varphi\in C_c^\infty((0,T)\times \Omega)$, we have
\ba\nn
\langle \pa_t\widetilde\phi_\e, \varphi \rangle=&
-\int_0^T\int_{\Omega} M(\widetilde\phi_\e) \nabla\widetilde\mu_\e\cdot\nabla \varphi \dd x \dd t
+\int_0^T\int_{\Omega}  \widetilde\phi_\e \widetilde\uu_\e\cdot\nabla \varphi \dd x \dd t
+\langle \pa_t\widetilde\phi_\e, (1-g_\e)\varphi \rangle\\
&\quad+\int_0^T\int_{\Omega} M(\widetilde\phi_\e) \nabla\widetilde\mu_\e\cdot\nabla((1-g_\e) \varphi )\dd x \dd t
-\int_0^T\int_{\Omega}  \widetilde\phi_\e \widetilde\uu_\e\cdot\nabla ((1-g_\e) \varphi ) \dd x \dd t.
\ea
Decompose $\widetilde\phi_\e$ into
\be\nn
\widetilde\phi_\e=\widetilde\phi_\e^{(1)}+\widetilde\phi_\e^{(2)},\quad \widetilde\phi_\e^{(1)}=g_\e \widetilde\phi_\e,\quad \widetilde\phi_\e^{(2)}=(1-g_\e) \widetilde\phi_\e.
\ee
Since $\langle\pa_t\widetilde\phi_\e^{(2)}, \varphi \rangle=\langle \pa_t\widetilde\phi_\e, (1-g_\e)\varphi \rangle$, 
we have
\ba\nn
\langle \pa_t\widetilde\phi_\e^{(1)}, \varphi \rangle=&
-\int_0^T\int_{\Omega} M(\widetilde\phi_\e) \nabla\widetilde\mu_\e\cdot\nabla \varphi \dd x \dd t
+\int_0^T\int_{\Omega}  \widetilde\phi_\e \widetilde\uu_\e\cdot\nabla \varphi \dd x \dd t\\
&\quad+\int_0^T\int_{\Omega} M(\widetilde\phi_\e) \nabla\widetilde\mu_\e\cdot\nabla((1-g_\e) \varphi )\dd x \dd t
-\int_0^T\int_{\Omega}  \widetilde\phi_\e \widetilde\uu_\e\cdot\nabla ((1-g_\e) \varphi ) \dd x \dd t.
\ea
By \eqref{uni-bound-tilde} and Lemmas \ref{Poincare inequality}, \ref{Korn inequality}, together with Sobolev embedding, we can deduce
\ba\nn
&\left| \int_0^T\int_{\Omega_\e} M(\widetilde\phi_\e) \nabla\widetilde\mu_\e\cdot\nabla \varphi \dd x \dd t\right| 
\leq M_2 \|\nabla\widetilde\mu_\e\|_{L^2 L^2} \|\nabla\varphi\|_{L^2 L^2} 
\leq C\|\widetilde\mu_\e\|_{L^2 H^1} \|\varphi\|_{L^2 H^1},\\
&\left| \int_0^T\int_{\Omega_\e}  \widetilde\phi_\e \widetilde\uu_\e\cdot\nabla \varphi \dd x \dd t\right| 
\leq  \|\widetilde\phi_\e\|_{L^\infty L^3} \|\widetilde\vu_\e\|_{L^2 L^6} \| \nabla \varphi\|_{L^2 L^2} 
\leq C \|\widetilde\phi_\e\|_{L^\infty H^1} \|\widetilde\vu_\e\|_{L^2 H^1} \| \varphi\|_{L^2 H^1}.
\ea
By virtue of \eqref{J23}, we obtain
\be\nn
|\langle \pa_t\widetilde\phi_\e^{(1)}, \varphi \rangle|
\leq C\| \varphi\|_{L^2 W^{1,r_1}} ,
\ee
which implies that $\pa_t \widetilde\phi_\e^{(1)}$ is uniformly bounded in $L^{2}(0,T; W^{-1,r_1'}(\Omega))$.

Since $\alpha>3$, we have
\be\nn
\gamma:=\frac{3\alpha-3}{2}-\alpha=\frac{\alpha-3}{2}>0.
\ee
Then we use \eqref{est-g}  to obtain
\ba\label{phi e 2}
\|\widetilde\phi_\e^{(2)}\|_{L^\infty H^1}&=\|(1-g_\e)\widetilde\phi_\e\|_{L^\infty H^1}\\
&\leq C (\|1-g_\e\|_{L^2}+\|\nabla g_\e\|_{L^2})(\|\widetilde\phi_\e\|_{L^\infty L^2}+\|\nabla\widetilde\phi_\e\|_{L^\infty L^2})\\
&\leq C\e^\gamma \|\widetilde\phi_\e\|_{L^\infty H^1}.
\ea
By \eqref{uni-bound-tilde}, we know that $\widetilde\phi_\e$ is uniformly bounded in $L^\infty(0,T; H^1(\Omega))$. Together with \eqref{phi e 2}, we can deduce that $\widetilde\phi_\e^{(1)}=\widetilde\phi_\e-\widetilde\phi_\e^{(2)}$ is uniformly bounded in $L^\infty(0,T; H^1(\Omega))$.

Recalling the fact that $\pa_t \widetilde\phi_\e^{(1)}$ is uniformly bounded in $L^{2}(0,T; W^{-1,r_1'}(\Omega))$,
and the compact embedding $H^1(\Omega)\hookrightarrow\hookrightarrow L^2(\Omega)\hookrightarrow W^{-1,r_1'}(\Omega)$, applying the Aubin--Lions--Simon compactness theorem (see \cite{Simon}), we obtain
\be\nn
\widetilde\phi_\e^{(1)} \to \phi  \ \mbox{strongly in}\ L^2(0,T; L^2(\Omega)).
\ee
Combining \eqref{uni-bound-tilde} and \eqref{phi e 2}, we can deduce
\be\nn
\widetilde\phi_\e^{(2)} \to 0 \ \mbox{strongly in}\ L^2(0,T; L^2(\Omega)),
\ee
and thus
\be\label{tilde phi strongly L2L2}
\widetilde\phi_\e \to \phi \ \mbox{strongly in}\ L^2(0,T; L^2(\Omega)).
\ee

\subsection{Homogenization process}\label{Homogenization process}
Once the convergence of the convective term and the strong convergence of $\widetilde \phi_\e$ are established, we can pass $\e\to 0$ to derive the limit system.

 Let us start with the momentum equations.
By Proposition \ref{moment-equa}, for ${\bm\varphi}\in C_c^\infty([0,T)\times \Omega;\mathbb{R}^3)$ with $\dive {\bm\varphi}=0$, there holds
\ba\label{Homo-G}
\l\GG_\varepsilon, \bvp \r=&
-\int_{0}^T\int_{\Omega}   \widetilde\uu_\e \cdot \pa_t \bm\varphi \dd x \dd t
-\int_0^T\int_{\Omega}  \widetilde\uu_\e\otimes\widetilde\uu_\e:\nabla \bm\varphi \dd x \dd t+
\int_0^T\int_{\Omega} \nu (\widetilde\phi_\e) D(\widetilde\uu_\e):D(\bm\varphi)\dd x \dd t\\ 
&\quad
+\int_0^T\int_{\Omega} \lambda_\e \widetilde\phi_\e \nabla \widetilde\mu_\e\cdot \bm\varphi\dd x \dd t
-\int_0^T\int_{\Omega} \mathbf g_\e\cdot \bm\varphi \dd x \dd t
-\int_{\Omega}\widetilde\uu_\e^0 \cdot \bm\varphi(0) \dd x, 
\ea
with
\be\nn
\left|\l\GG_\varepsilon, \bvp \r \right| \leq C\e^{\sigma} \lambda_\e^{1\over2}(\|\pa_t \bvp\|_{L^{4\over3} L^2}+ \|\nabla\bvp\|_{L^4 L^{r_1}}+ \|\bvp(0)\|_{L^{r_2}}).
\ee
By virtue of the Lipschitz continuity of $\nu$, together with \eqref{tilde phi strongly L2L2}, we can deduce
\be\label{nu-phi-strong}
\nu (\widetilde\phi_\e) \to \nu (\phi) \quad \text{strongly in } L^2(0,T;L^2(\Omega)).
\ee
Combining \eqref{assumption-lambda}--\eqref{assumption-g}, \eqref{Convergence of u}, \eqref{convective term}, \eqref{tilde phi strongly L2L2} and \eqref{nu-phi-strong}, letting $\e\to 0$ in \eqref{Homo-G}, we obtain
\ba\label{lim-momentum-eq}
0=&-\int_{0}^T\int_{\Omega}  \uu \cdot \pa_t \bm\varphi \dd x \dd t
-\int_0^T\int_{\Omega}  \uu\otimes\uu:\nabla \bm\varphi \dd x \dd t+
\int_0^T\int_{\Omega} \nu (\phi) D(\uu):D(\bm\varphi)\dd x \dd t\\ 
&\quad
+\int_0^T\int_{\Omega} \lambda \phi \nabla \mu\cdot \bm\varphi\dd x \dd t
-\int_0^T\int_{\Omega} \sqrt{\lambda}\mathbf g\cdot \bm\varphi \dd x \dd t
-\int_{\Omega}\sqrt{\lambda}\uu_0 \cdot \bm\varphi(0) \dd x. 
\ea

Next, we study the limit process of phase-field equation. By Proposition \ref{phase-field-equa}, for any $\varphi\in C_c^\infty([0,T)\times \Omega)$, there holds
\ba\label{Homo-H}
&-\int_0^T\int_{\Omega}  \widetilde\phi_\e\pa_t\varphi \dd x\dd t+
\int_0^T\int_{\Omega}  M(\widetilde\phi_\e)\nabla\widetilde\mu_\e\cdot\nabla \varphi \dd x \dd t\\
&\quad=\int_0^T\int_{\Omega}  \widetilde\phi_\e \widetilde\uu_\e\cdot\nabla \varphi \dd x \dd t
+\int_{\Omega}  \phi_\e^0\varphi(0) \dd x
+\l\HH_\varepsilon, \varphi \r,
\ea
with
\begin{equation}\nn
\left|\l\HH_\varepsilon, \varphi \r\right| \leq C\e^{\sigma}(\|\pa_t\varphi\|_{L^2 L^{3\over2}}+\|\nabla\varphi\|_{L^4 L^{r_1}}+ \|\varphi(0)\|_{L^{3}}).
\end{equation}
By virtue of the Lipschitz continuity of $M$, together with \eqref{tilde phi strongly L2L2}, we can deduce
\be\label{M-phi-strong}
M (\widetilde\phi_\e) \to M (\phi) \quad \text{strongly in } L^2(0,T;L^2(\Omega)).
\ee
Combining \eqref{initial-u-phi}, \eqref{Convergence of u}, \eqref{tilde phi strongly L2L2} and \eqref{M-phi-strong}, we send $\e\to 0$ in \eqref{Homo-H} to obtain
\ba \label{lim-phase-field-eq}
&-\int_0^T\int_{\Omega} \phi\pa_t\varphi \dd x\dd t+
\int_0^T\int_{\Omega} M(\phi) \nabla\mu\cdot\nabla \varphi \dd x \dd t
=\int_0^T\int_{\Omega}  \phi \uu\cdot\nabla \varphi \dd x \dd t
+\int_{\Omega}  \phi_0\varphi(0) \dd x.
\ea

Finally, we investigate the limit process of chemical-potential equation. By Proposition \ref{chemical-equa}, for any $\psi\in C_c^\infty((0,T)\times \Omega)$, there holds
\ba\label{Homo-K}
\int_0^T\int_{ \Omega}  \widetilde\mu_\e \psi \dd x \dd t
-\int_0^T\int_{ \Omega} \nabla\widetilde\phi_\e \cdot \nabla\psi \dd x\dd t
=\int_0^T\int_{ \Omega}  F'(\widetilde\phi_\e) \psi \dd x\dd t+\l\KK_\varepsilon, \psi \r,
\ea
with
\begin{equation}\nn
|\l\KK_\varepsilon, \psi \r| \leq C \e^{\sigma}\|\nabla\psi\|_{L^2L^{r_1}}.
\end{equation}
By \eqref{Convergence of u} and \eqref{tilde phi strongly L2L2}, one has
\ba\nn
\left|\int_0^T \int_{\Omega}  (F'(\widetilde \phi_\e)-F'( \phi))\psi \dd x \dd t \right|
&\leq C\|(\widetilde \phi_\e-\phi)(\widetilde \phi_\e^2+\widetilde \phi_\e \phi+\phi^2-1)\|_{ L^2 L^{\frac{6}{5}}} \|\psi\|_{L^2 L^6}\\
&\leq C\|\widetilde \phi_\e-\phi\|_{L^2 L^2} (\|\widetilde \phi_\e\|_{L^\infty L^6}^2+\|\phi\|_{L^\infty L^6}^2+1) \|\psi\|_{L^2 L^6}
\to 0.
\ea
By virtue of \eqref{Convergence of u}, we pass to the limit $\e\to 0$ in \eqref{Homo-K} to deduce
\ba\label{lim-chemical-eq}
\int_0^T\int_{ \Omega}  \mu \psi \dd x \dd t
-\int_0^T\int_{ \Omega} \nabla\phi \cdot \nabla\psi \dd x\dd t
=\int_0^T\int_{ \Omega}  F'(\phi) \psi \dd x\dd t. 
\ea

Combining \eqref{Convergence of u}, \eqref{strong convergence of u}, \eqref{tilde phi strongly L2L2}, \eqref{lim-momentum-eq}, \eqref{lim-phase-field-eq} and \eqref{lim-chemical-eq}, 
 we complete the proof of Theorem \ref{Main theorem}.
 
 \subsection{Case $\lambda=0$}\label{lambda0}
Finally, we study the case $\lambda_\e\to\lambda=0$ and prove Corollary \ref{corollary}. All above arguments remain valid in this case. Notably, \eqref{uni-bound-tilde} implies $\widetilde\uu_\e\to \bf{0}$ in $L^{2}(0,T;L^{2}(\Omega;\mathbb{R}^3))$. The limit phase-field equation reduces to \eqref{lim-phase-field-eq} with $\uu=\bf{0}$, while the limit chemical-potential equation \eqref{lim-chemical-eq} stays unchanged. However, every term in \eqref{lim-momentum-eq} vanishes identically. This observation motivates us to identify a suitable scaling that yields a nontrivial limiting velocity.
 
 When $\lambda_\e\to\lambda=0$, by virtue of \eqref{uni-bound-tilde}, we deduce that $\uu=\bf 0$, and thus the limit phase-field equation \eqref{lim-phase-field-eq} reads
 \ba \label{lim-phase-field-eq-coro}
&-\int_0^T\int_{\Omega} \phi\pa_t\varphi \dd x\dd t+
\int_0^T\int_{\Omega} M(\phi) \nabla\mu\cdot\nabla \varphi \dd x \dd t
=\int_{\Omega}  \phi_0\varphi(0) \dd x.
\ea
The limit chemical-potential equation \eqref{lim-chemical-eq} remains unchanged.

Under the scaling $\ww_\e=\frac{\uu_\e}{\sqrt{\lambda_\e}}$, we have
\be\nn
 \ww_\e=\bm{ 0} \text{ on } \pa\Omega_\e,\quad \text{and}\quad \dive \ww_\e=0 \text{ in } \Omega_\e.
 \ee
  The zero extension of $\ww_\e$ is $\widetilde\ww_\e=\frac{\widetilde\uu_\e}{\sqrt{\lambda_\e}}$. Estimates \eqref{uni-bound-tilde} implies
\ba\nn
\|\widetilde\ww_\e\|_{L^\infty L^2}+\|D(\widetilde\ww_\e)\|_{L^2 L^2}\leq C.
\ea
Up to extraction of a subsequence, the following weak convergences hold:
\ba \label{Convergence of u-coro}
&\widetilde\ww_\e \to \ww \ \mbox{weakly-(*) in}\ L^\infty(0,T; L^2(\Omega; \R^3)) \ \mbox{and weakly in}\ L^2(0,T; H^1_0(\Omega;\R^3)).
\ea
The divergence free condition $\dive \ww=0$ follows from $\dive \widetilde\ww_\e=0$. Moreover, estimate \eqref{pat u varphi} implies
\ba\label{pat w varphi}
|\langle\pa_t \widetilde\ww_\e , \bm\varphi\rangle |\leq C (\|\bvp\|_{L^4 W^{1,r_1}}+\e^\sigma \|\pa_t \bvp\|_{L^{4\over3} L^2}),
\ea
for all ${\bm\varphi}\in C_c^\infty((0,T)\times \Omega;\mathbb{R}^3)$.
Thus we have the decomposition
\ba\nn
\widetilde\ww_\e=\widetilde\ww_\e^{(1)}+\e^\sigma\widetilde\ww_\e^{(2)},
\ea
where $\pa_t \widetilde\ww_\e^{(1)}$ is uniformly bounded in $L^{4\over3}(0,T; W^{-1,r_1'}(\Omega;\mathbb{R}^3))$, and $\widetilde\ww_\e^{(2)}$ is uniformly bounded in $L^4(0,T;L^2(\Omega;\mathbb{R}^3))$.
Statements \eqref{pat u varphi}--\eqref{strong convergence of u}
hold true when we replace $\widetilde\vu_\e,\widetilde\vu_\e^{(1)},\widetilde\vu_\e^{(2)}$ by $\widetilde\ww_\e,\widetilde\ww_\e^{(1)},\widetilde\ww_\e^{(2)}$. As a result, we derive
\be\label{strong convergence of w}
\widetilde\ww_\e \to \ww \text{ in } L^{2}(0,T;L^{2}(\Omega)).
\ee
 
By Proposition \ref{moment-equa}, for any ${\bm\varphi}\in C_c^\infty([0,T)\times \Omega;\mathbb{R}^3)$ with $\dive {\bm\varphi}=0$, there holds
\ba\label{momentum eq-1-coro}
&-\int_{0}^T\int_{\Omega}  \widetilde\ww_\e\cdot \pa_t \bm\varphi \dd x \dd t
-\int_0^T\int_{\Omega}  \sqrt{\lambda_\e}  \widetilde\ww_\e\otimes\widetilde\ww_\e:\nabla \bm\varphi \dd x \dd t
+\int_0^T\int_{\Omega} \nu (\widetilde\phi_\e) D(\widetilde\ww_\e):D(\bm\varphi)\dd x \dd t\\
&\quad+\int_0^T\int_{\Omega} \sqrt{\lambda_\e} \widetilde\phi_\e \nabla \widetilde\mu_\e\cdot \bm\varphi\dd x \dd t
=\int_0^T\int_{\Omega} \frac{\mathbf g_\e}{\sqrt{\lambda_\e}}\cdot \bm\varphi \dd x \dd t
+\int_{\Omega}  \widetilde\ww_\e \cdot\bm\varphi (0)\dd x
+{\lambda_\e}^{-\frac{1}{2}} \l\GG_\varepsilon, \bvp \r, 
\ea
with
\begin{equation}\label{est-GG-coro}
\left|{\lambda_\e}^{-\frac{1}{2}} \l\GG_\varepsilon, \bvp \r \right| 
\leq C\e^{\sigma} (\|\pa_t \bvp\|_{L^{4\over3} L^2}+ \|\nabla\bvp\|_{L^4 L^{r_1}}+ \|\bvp(0)\|_{L^{r_2}}).
\end{equation}
Combining \eqref{assumption-lambda}--\eqref{assumption-g}, \eqref{Convergence of u}, \eqref{tilde phi strongly L2L2}, \eqref{Convergence of u-coro}, \eqref{strong convergence of w} and \eqref{est-GG-coro}, letting $\e\to 0$ in \eqref{momentum eq-1-coro}, we deduce
\ba\label{lim-momentum-eq-coro}
-\int_{0}^T\int_{\Omega}  \ww\cdot \pa_t \bm\varphi \dd x \dd t
+\int_0^T\int_{\Omega} \nu (\phi) D(\ww):D(\bm\varphi)\dd x \dd t
=\int_0^T\int_{\Omega} \mathbf g\cdot \bm\varphi \dd x \dd t
+\int_{\Omega} \vu_0 \cdot\bm\varphi (0)\dd x. 
\ea
Combining \eqref{Convergence of u}, \eqref{tilde phi strongly L2L2}, \eqref{lim-chemical-eq}--\eqref{Convergence of u-coro}, \eqref{strong convergence of w} and \eqref{lim-momentum-eq-coro}, we complete the proof of Corollary \ref{corollary}.

	\paragraph{Conflict of interest}
	The authors declare no conflict of interest in this paper.

\end{document}